\documentclass[10pt]{amsart}
\usepackage[a4paper]{geometry}
\title{From the function-sheaf dictionary to quasicharacters of $p$-adic tori}

\author{Clifton Cunningham}
\author{David Roe}

\address{Department of Mathematics and Statistics, University of Calgary, 2500 University Drive Northwest, Calgary, Alberta, Canada, {T2N~1N4}.}
\email{cunning@math.ucalgary.ca}
\address{Department of Mathematics, University of British Columbia, 1984 Mathematics Road, Vancouver, British Columbia, Canada, {V6T~1Z2}.}
\email{roed.math@gmail.com}



\usepackage{amssymb}
\usepackage[initials]{amsrefs}

\BibSpec{lecture}{%
  +{} {\PrintAuthors}				{author}
  +{,}{ }						{title}
  +{}{ \PrintDateB}				{date}
  +{.}{ Lecture delivered at \textit}	{conference}
  +{,}{ }						{address}
  +{.}{ }						{transition}
}

\BibSpec{book}{%
    +{}  {\PrintPrimary}				{transition}
    +{,} { }							{title}
    +{.} { }							{part}
    +{:} { }							{subtitle}
    +{,} { \PrintEdition}					{edition}
    +{}  { \PrintEditorsB}				{editor}
    +{,} { \PrintTranslatorsC}			{translator}
    +{,} { \PrintContributions}			{contribution}
    +{,} { }							{series}
    +{} { \textbf}						{volume}
    +{,} { }							{publisher}
    +{,} { } 							{organization}
    +{,} { }							{address}
    +{,} { \PrintDateB}					{date}
    +{,} { }							{status}
    +{}  { \parenthesize}				{language}
    +{}  { \PrintTranslation}				{translation}
    +{;} { \PrintReprint}				{reprint}
    +{.} { }							{note}
    +{.} {}							{transition}
    +{}  {\SentenceSpace \PrintReviews}	{review}
}

\BibSpec{collection.article}{%
    +{}  {\PrintAuthors}				{author}
    +{,} { \textit}						{title}
    +{.} { }							{part}
    +{:} { \textit}						{subtitle}
    +{,} { \PrintContributions}			{contribution}
    +{,} { \PrintConference}				{conference}
    +{}  {\PrintBook}					{book}
    +{,} { }							{booktitle}
    +{}  { \PrintEditorsB}				{editor}
    +{,} { }							{series}
    +{} { \textbf}						{volume}
    +{,} { }							{publisher}
    +{,} { }							{address}
    +{,} { \PrintDateB}					{date}
    +{,} { }							{pages}
    +{,} { }							{status}
    +{,} { \PrintDOI}					{doi}
    +{,} { available at \eprint}			{eprint}
    +{}  { \parenthesize}				{language}
    +{}  { \PrintTranslation}				{translation}
    +{;} { \PrintReprint}				{reprint}
    +{.} { }							{note}
    +{.} {}							{transition}
    +{}  {\SentenceSpace \PrintReviews}	{review}
}

\usepackage{mathrsfs}
\usepackage{enumitem}
\usepackage{tikz}
\usetikzlibrary{shapes,arrows,calc,matrix}
\usepackage{tikz-cd}
\usepackage{hyperref}

\theoremstyle{plain}
      \newtheorem{theorem}{Theorem}[section]
      \newtheorem*{theorem*}{Theorem}
      \newtheorem{proposition}[theorem]{Proposition}
      \newtheorem{lemma}[theorem]{Lemma}
      \newtheorem{corollary}[theorem]{Corollary}

      \theoremstyle{definition}
      \newtheorem{definition}[theorem]{Definition}

      \newtheorem{remark}[theorem]{Remark}
      \newtheorem{example}[theorem]{Example}

\tikzset{every picture/.style={>=stealth},label/.style={font=\footnotesize}}

\newcommand{\ZZ}{{\mathbb{Z}}}
\newcommand{\NN}{{\mathbb{N}}}

\newcommand{\EE}{\mathbb{\bar Q}_\ell}
\newcommand{\OK}{\mathcal{O}_K}
\newcommand{\OL}{\mathcal{O}_L}
\newcommand{\OO}[1]{\mathcal{O}_{#1}}
\newcommand{\bFq}{\bar{k}}
\newcommand{\Fq}{k}

\newcommand{\EEx}{\EE^\times}
\newcommand{\ZEx}{\mathbb{\bar Z}_\ell^\times}
\newcommand{\Weil}[1]{\mathcal{W}_{#1}}
\newcommand{\m}{{\mathfrak{m}}}
\newcommand{\Gm}[1]{\mathbb{G}_{\hskip-1pt\textbf{m},#1}}

\DeclareMathOperator{\Gal}{Gal}
\newcommand{\Frob}[1]{\operatorname{Fr}_{#1}}
\DeclareMathOperator{\Aut}{Aut}
\DeclareMathOperator{\Hom}{Hom}

\DeclareMathOperator{\coker}{coker}
\DeclareMathOperator{\Gr}{Gr}

\DeclareMathOperator{\id}{id}
\DeclareMathOperator{\Ext}{Ext}
\DeclareMathOperator{\Hh}{H}
\DeclareMathOperator{\Res}{Res}
\DeclareMathOperator{\Nm}{Nm}
\DeclareMathOperator{\trace}{Tr}

\DeclareMathOperator{\Lang}{Lang}

\DeclareMathOperator{\Tor}{Tor}

\newcommand{\Spec}[1]{{\operatorname{Spec}(#1)}}
\newcommand{\sheafHom}{{\mathscr{H}\hskip-4pt{\it o}\hskip-2pt{\it m}}}

\newcommand{\ceq}{{\, :=\, }}
\newcommand{\tq}{{\ \vert\ }}
\newcommand{\iso}{{\ \cong\ }}
\newcommand{\trFrob}[1]{t_{#1}}
\newcommand{\TrFrob}[1]{\operatorname{Tr}_{#1}}

\newcommand{\cs}[1]{{\mathcal{#1}}}
\newcommand{\gcs}[1]{{\mathcal{\bar #1}}}

\newcommand{\CS}{{\mathcal{C\hskip-0.8pt S}}}
\newcommand{\bCS}{{\CS_0}}
\newcommand{\CSiso}[1]{\CS(#1)_{/\text{iso}}}
\newcommand{\bCSiso}[1]{\bCS(#1)_{/\text{iso}}}
\newcommand{\catname}[1]{\normalfont{\textsf{#1}}}
\newcommand{\Sch}[1]{{\catname{Sch}_{/#1}}}
\newcommand{\QCS}{{\mathcal{QC\hskip-0.8pt S}}}
\newcommand{\QCSiso}[1]{\QCS(#1)_{/\text{iso}}}
\makeatletter
\newcommand{\labitem}[2]{
\def\@itemlabel{\textbf{#1}}
\item
\def\@currentlabel{#1}\label{#2}}
\makeatother
\renewcommand{\bf}{\bar{f}}
\newcommand{\bg}{{\bar{g}}}
\newcommand{\bm}{\bar{m}}
\newcommand{\bG}{\bar{G}}
\newcommand{\bH}{\bar{H}}
\newcommand{\brho}{{\bar\rho}}

\newcommand{\tight}[3]{\hspace{-#1pt}{#2}\hspace{-#3pt}}

\newcommand{\bGxG}{\text{$\bar{G} \tight{1}{\times}{1} \bar{G}$}}
\newcommand{\bfxf}{\text{$\bar{f} \tight{1}{\times}{1} \bar{f}$}}
\newcommand{\GxxG}{\text{$G \tight{1}{\times}{1} G$}}
\newcommand{\LxL}{\text{$\gcs{L} \tight{0}{\boxtimes}{0} \gcs{L}$}}

\begin{document}

\begin{abstract}
We consider the rigid monoidal category of character sheaves on a smooth commutative group scheme $G$ over a finite field
$k$ and expand the scope of the function-sheaf dictionary from connected commutative algebraic groups to this setting.
We find the group of isomorphism classes of character sheaves on $G$ and show that it is an extension of the group
of characters of $G(k)$ by a cohomology group determined by the component group scheme of $G$.
We also classify all morphisms in the category character sheaves on $G$.
As an application, we study character sheaves on Greenberg transforms of locally finite type N\'eron models of algebraic tori over local fields. 
This provides a geometrization of quasicharacters of $p$-adic tori.
\end{abstract}

\maketitle

\section*{Introduction}

As Deligne explained in \cite{deligne:SGA4.5}*{Sommes trig.}, if $G$ is a connected commutative algebraic group
over a finite field $k$, then the trace of Frobenius provides a bijection between the group $G(\Fq)^*$ of $\ell$-adic
characters of $G(\Fq)$ and isomorphism classes of those rank-one $\ell$-adic local systems $\mathcal{E}$ on $G$ for which 
\begin{equation}\label{introbox}
m^* \cs{E} \iso \cs{E} \boxtimes \cs{E},
\end{equation}
where $m : G\times G\to G$ is the multiplication map.
If one wishes to make a category from this class of local systems, one is led to consider morphisms
$\cs{E} \to \cs{E}'$ of sheaves which are compatible with particular choices of \eqref{introbox} for $\cs{E}$ and $\cs{E'}$. 
{\it A priori}, the composition $\cs{E} \to \cs{E}' \to \cs{E}''$ of two such morphisms need not be
compatible with the choices of \eqref{introbox} for $\cs{E}$ and $\cs{E}''$.
 However, for connected $G$ the isomorphism \eqref{introbox} is unique, if it exists, and there is no impediment
 to making the dictionary categorical.

If $G$ is a commutative algebraic group over $\Fq$ which is not connected, however, then the isomorphism \eqref{introbox}
need not be unique. In order to track the choice of isomorphism, consider the category $\bCS(G)$ of pairs $(\cs{E},\mu_\cs{E})$
where $\cs{E}$ is a rank-one local system on $G$ and $\mu_\cs{E} : m^*\cs{E} \to \cs{E}\boxtimes\cs{E}$ is a chosen
isomorphism of local systems on $G\times G$. 
In this case, the trace of Frobenius provides an epimorphism from isomorphism classes of objects in $\bCS(G)$ to characters
of $G(\Fq)$, but the epimomorphism need not be injective; consequently,
every character of $G(\Fq)$ may be geometrized as a pair $(\cs{E},\mu_\cs{E})$ but perhaps not uniquely.
Indeed, it follows from a special case of the main result of this paper that the kernel of the trace of Frobenius
$\bCS(G) \to G(\Fq)^*$ trivial if and only if the  group scheme of connected components of $G$ is cyclic.
The defect in the function-sheaf dictionary for characters of commutative algebraic groups over finite fields may be
addressed with the following observation: if $(\cs{E},\mu_\cs{E})$ and $(\cs{E}',\mu_\cs{E}')$ determine the same
character of $G(\Fq)$ then $\cs{E}\iso \cs{E}'$ as local systems on $G$.

Motivated by an application to quasicharacters of algebraic tori over local fields, in this paper we extend the
function-sheaf dictionary from commutative algebraic groups over finite fields to smooth commutative group schemes $G$ over $\Fq$.
In order to do this, we replace the local system $\cs{E}$ on $G$ with a Weil local system while retaining the extra structure $\mu_\cs{E}$. 
In this way we are led to the category $\CS(G)$ of {\it character sheaves} on $G$ (\S\ref{ssec:category}):
objects in $\CS(G)$ are triples $(\gcs{L}, \mu,\phi)$, where $\gcs{L}$ is a rank-one local system on $\bG := G\times_{\Spec{\Fq}} \Spec{\bFq}$
and $\phi : \Frob{G}^* \gcs{L}\to \gcs{L}$ and $\mu : \bm^* \gcs{L} \iso \gcs{L} \boxtimes \gcs{L}$ are
isomorphisms of sheaves satisfying certain compatibility conditions;
morphisms in $\CS(G)$ are then morphisms of Weil sheaves which are compatible with the extra structure.
This paper establishes the basic properties of category $\CS(G)$, using the group homomorphism 
\[
\TrFrob{G}: \CSiso{G} \to G(\Fq)^*
\]
provided by the trace of Frobenius  to find the relation between character sheaves on $G$ and characters of $G(\Fq)$.
Then we return to our motivating application and use character sheaves to geometrize and categorify quasicharacters algebraic tori over local fields.

We begin our study of category $\CS(G)$ by returning to the case when $G$ is a
connected commutative algebraic group over $\Fq$, revisiting Deligne's function-sheaf dictionary (\S\ref{sec:category}).
We consider character sheaves that arise via base change to $\bFq$ from local systems on $G$ (\S\ref{ssec:descentG}) and
those that appear in a pushforward from a constant sheaf along a discrete isogeny $H \to G$ (\S\ref{ssec:discrete_isogenies}).  
While these constructions make sense even for non-connected $G$, in the connected case we show that
every character sheaf can be described in both of these ways (\S\ref{ssec:connected}).  
We use this fact to prove that $\TrFrob{G} : \CS(G) \to G(\Fq)^*$ is an isomorphism for connected commutative algebraic groups $G$.
We also determine the automorphism groups of character sheaves on such $G$. 
These facts are well known.

Next, we consider character sheaves on \'etale commutative group schemes $G$ over $\Fq$ (\S\ref{sec:etale}).
\'Etale group schemes form a counterpoint to connected algebraic groups, since the component group of any
smooth group scheme is an \'etale group scheme.
Our key tools for understanding the trace of Frobenius in the \'etale case 
are a reinterpretation of $\CS(G)$ in terms of stalks (\S\ref{ssec:stalks}) and 
the Hochschild-Serre spectral sequence (\S\ref{ssec:E}) for $\Weil{} \ltimes \bG$, where $\Weil{} \subset \Gal(\Fq/\Fq)$ is the Weil group for $\Fq$.
We define (\S\ref{ssec:S}) an isomorphism $S_G$ from $\CSiso{G}$ to the second cohomology of the total space of the spectral sequence
\[
E_2^{p,q} \ceq \Hh^p(\Weil{}, \Hh^q(\bG, \EEx)) \Rightarrow \Hh^{p+q}(\Weil{} \ltimes \bG, \EEx).
\]
Paired with the short exact sequence
\[
  0 \to \Hh^0(\Weil{},\Hh^2(\bG,\EEx)) \to \Hh^2(E^\bullet_G) \to \Hh^1(\Weil{},\Hh^1(\bG,\EEx)) \to 0
\]
arising from the spectral sequence, the isomorphism $S_G$ allows us to show (\S\ref{ssec:SandT})
that the group homomorphism $\TrFrob{G}: \CSiso{G} \to G(\Fq)^*$ is surjective with kernel $\Hh^2(\bG,\EEx)^{\Weil{}}$.
The necessity of using Weil local systems on $G$ in the definition of $\CS(G)$ already appears here:
if one were to use local systems on $G$ instead, the group homomorphism $\TrFrob{G}$ would not then be surjective (\S \ref{ssec:bS}). 
Moreover, as examples show (\S \ref{eg:}), the kernel of $\TrFrob{G}$ is non-trivial in general.

Having understood $\CS(G)$ in two extreme cases --
for connected commutative algebraic groups and for \'etale commutative group schemes --
we turn to the case of smooth commutative group schemes (\S\ref{sec:main}) using the component group sequence
\[
0 \to G^0 \to G \to \pi_0(G) \to 0.
\]
Using pullbacks of character sheaves we obtain a diagram
\[
  \begin{tikzcd}[row sep=20, column sep=20]
    0 \rar & \CSiso{\pi_0(G)} \rar \dar{\TrFrob{\pi_0(G)}}
    & \CSiso{G} \rar \dar{\TrFrob{G}} & \CSiso{G^0} \rar \dar{\TrFrob{G^0}} & 0\\
    0 \rar & \pi_0(G)(\Fq)^* \rar & G(\Fq)^* \rar & G^0(\Fq)^* \rar & 0.
  \end{tikzcd}
\]
We show that the rows of this diagram are exact (\S\S\ref{ssec:restriction},\ref{ssec:component}),
so we may apply the snake lemma to prove the main theorem of the paper,
\begin{theorem*}[{Thm. \ref{thm:snake}}]
If $G$ is a smooth commutative group scheme over $\Fq$ then the trace of Frobenius gives a short exact sequence
\[
\begin{tikzcd}
0 \arrow{r} & \Hh^2(\pi_0(\bG),\EEx)^{\Weil{}} \arrow{r} & \CSiso{G} \arrow{r}{\TrFrob{G}} & G(\Fq)^* \arrow{r} & 0.
\end{tikzcd}
\]
\end{theorem*}
If the component group scheme $\pi_0(\bG)$ is cyclic, then the kernel of $\TrFrob{G}$ will be trivial and
each character of $G(\Fq)$ will uniquely determine a character sheaf on $G$, up to isomorphism.  But when
$\pi_0(\bG)$ is large (c.f. Remark \ref{rem:H2}), $G$ will admit \emph{invisible character sheaves} with
trivial trace of Frobenius.

We also illuminate the nature of the category $\CS(G)$ by showing that every morphism in this category
is either an isomorphism or trivial, and by showing
\begin{theorem*}[{Thm. \ref{thm:autornaught}}]
If $G$ is a smooth commutative group scheme over $\Fq$ then
\[
\Aut(\cs{L}) \iso  \Hh^1(\pi_0(\bG), \EEx)^{\Weil{}}
\]
for all quasicharacter sheaves $\cs{L}$ on $G$.
\end{theorem*}

\subsection*{Application to quasicharacters of \texorpdfstring{$p$}{p}-adic tori and abelian varieties}
As indicated above, our interest in the function-sheaf dictionary for smooth commutative group schemes
over finite fields comes from an application to $p$-adic representation theory,
specifically to quasicharacters (\S\ref{ssec:quasicharacters}) of $p$-adic tori.
However, we found that our method of passing from $p$-adic tori to group schemes over $\Fq$ applies more generally to
any local field $K$ with finite residue field $\Fq$ and to any commutative algebraic group over $K$ that admits a N\'eron model $X$.
This class of algebraic groups over $K$ includes abelian varieties and unipotent $K$-wound groups,
in addition to the algebraic tori we initially considered.

In this paper we show that if $X$ is as above then quasicharacters of $X(K)$ are geometrized and categorified by character sheaves on the
Greenberg transform $\Gr_R(X)$ of the N\'eron model $X$.
Although not locally of finite type, $\Gr_R(X)$ is a commutative group scheme over $\Fq$ and also a projective limit
of smooth commutative group schemes $\Gr^R_n(X)$.
This structure allows us to adapt our work on character sheaves on smooth group schemes over
finite fields to construct (\S\ref{ssec:CS_on_GN}) a category $\QCS(X)$
of {\it quasicharacter sheaves} for $X$, which are certain sheaves on $\Gr_R(X)\times_{\Spec{\Fq}} \Spec{\bFq}$, with extra structure.
The ability to generalize the function-sheaf dictionary to non-connected group schemes plays a crucial role in this application.

Having defined quasicharacter sheaves on N\'eron models of commutative algebraic groups over $K$
and character sheaves on commutative group schemes over $\Fq$, we consider how
these categories are related as $K$ and $\Fq$ vary.  We describe (\S\ref{ssec:basechange}) functors between categories
of quasicharacter sheaves that model restriction and norm homomorphisms of character groups $G(k')^* \to G(k)^*$ and $G(k)^* \to G(k')^*$,
and describe how quasicharacter sheaves behave under Weil restriction (\S\ref{ssec:wrK}).  We also give (\S\ref{ssec:transfer})
a categorical version of a result of Chai and Yu \cite{chai-yu:01a},
relating quasicharacter sheaves for tori over different local fields, even local fields
with different characteristic.

Finally, specializing to the case that $X = T$ is the N\'eron model of an algebraic torus over $K$ (\S\ref{ssec:CS_tori}), 
we give a canonical short exact sequence 
\[
0 \to \Hh^2(X_*(T)_{\mathcal{I}_K},\EEx)^{\Weil{}} \to \QCSiso{T} \to \Hom_\text{}(T(K),\EEx) \to 0,
\]
where $X_*(T)_{\mathcal{I}_K}$ is the group of coinvariants of the cocharacter lattice $X_*(T)$ of the algebraic torus
$T_K$ by the action of the inertia group $\mathcal{I}_K$ of $K$, and where $\Hom_\text{}(T(K),\EEx)$
denotes the group of quasicharacters of $T(K)$.
We further show that automorphism groups in $\QCS(T)$ are given, for every quasicharacter sheaf $\cs{F}$ for $T$, by
\[
\Aut(\cs{F}) \iso (\check{T}_\ell)^{\Weil{K}},
\]
where $\Weil{K}$ is the Weil group for $K$ and $\check{T}_\ell$ is the $\ell$-adic dual torus to $T$.

\subsection*{Relation to other work.}
The main use of the term character sheaf is of course due to Lusztig.
It is applied to certain perverse sheaves on connected reductive algebraic groups over algebraically closed fields in
\cite{lusztig:85a}*{Def.~2.10} and to certain perverse sheaves on certain reductive groups over algebraically closed fields in the series of papers
beginning with \cite{lusztig:disconnected1}.
When commutative, it is not difficult to relate Frobenius-stable character sheaves to our character sheaves
(Remark~\ref{rem:Lusztig}). The new features that we have found pertaining to Weil sheaves
and $\Hh^2(\pi_0(\bG),\EEx)^{\Weil{}}$ do not arise in that context because, for such groups,
Weil sheaves are unnecessary (\S \ref{ssec:alg_groups}) and $\Hh^2(\pi_0(\bG),\EEx)^{\Weil{}}=0$ (Remark \ref{rem:H2}).

For a connected commutative algebraic group over a finite field, it is not uncommon to refer to
local systems satisfying \eqref{introbox}  as character sheaves; see for example, \cite{kamgarpour:09a}*{Intro}.
Our definition of character sheaves on smooth commutative group schemes over finite fields evolved from this notion,
with an eye toward quasicharacters of $p$-adic groups.
The process of creating a category from the group of quasicharacters of a $p$-adic torus informs our choice of the
term quasicharacter sheaf in this paper.

We anticipate that future work on quasicharacter sheaves will make use of   \cite{suzuki-yoshida:12a} and \cite{suzuki:14a},
and will clarify the relation between this project and other attempts to geometrize admissible distributions on $p$-adic groups,
such as \cite{cunningham-kamgarpour:13a} (limited to quasicharacters of $\ZZ_p^\times$) and \cite{aubert-cunningham:02a}
(limited to characters of depth-zero representations). 
We are actively pursuing the question of how to extend the notion of quasicharacter sheaves to provide a geometrization
of admissible distributions on connected reductive algebraic groups over $p$-adic fields, not just commutative ones.

\subsection*{Acknowledgements.}

We thank Pramod Achar, Masoud Kamgarpour and Hadi Salmasian for 
allowing us to hijack much of a Research in Teams meeting at the Banff International Research Station into a discussion of
quasicharacter sheaves.  We also thank them for their kindness, knowledge and invaluable help.
We thank Takashi Suzuki for clarifying the relation between this project and his recent work and for very helpful
observations and suggestions, especially related to our use of the Hochschild-Serre spectral sequence.
We thank Alessandra Bertrapelle and Cristian Gonz\'ales-Avil\'es for disabusing us of a misapprehension concerning
the Greenberg realization functor and for drawing our attention to their result on Weil restriction and the Greenberg transform.
Finally, we thank Joseph Bernstein for very helpful comments.

Finally, we gratefully acknowledge the financial support of the Pacific Institute for the Mathematical Sciences
and the National Science and Engineering Research Council (Canada), as well the hospitality of the
Banff International Research Station during a Research in Teams program.

\tableofcontents

\section{Definitions and recollections} \label{sec:category}

\subsection{Notations}\label{ssec:notation}

Throughout this paper, $G$ is a smooth commutative group scheme
over a finite field $\Fq$ and $m : G \times G\to G$ is its multiplication morphism.

We will make use of the short exact sequence of smooth group schemes defining the component group scheme for $G$:
\[
\begin{tikzcd}
0 \rar & G^0 \arrow{r}{\iota_0} & G \arrow{r}{\pi_0} & \pi_0(G) \rar & 0.
\end{tikzcd}
\]
Then $G^0$ is a connected algebraic group and $\pi_0(G)$ is an \'etale commutative group scheme.
In contrast to the case of algebraic varieties, the component group scheme $\pi_0(G)$ for $G$ need not be finite.

It follows from the smoothness of $G$ that the structure morphism $G \to \Spec{\Fq}$ is locally of finite type, being smooth.
If the structure morphism $G \to \Spec{\Fq}$ is also \'etale, then $G$ is an \'etale group scheme; this does not imply that $\pi_0(G)$ is finite.
An algebraic group over $\Fq$ is a smooth group scheme of finite type, in which case its component group scheme is finite.

We fix an algebraic closure $\bFq$ of $\Fq$ and write $\bG$ for the
smooth commutative group scheme $G \times_{\Spec{\Fq}} \Spec{\bFq}$ over $\bFq$
obtained by base change from $k$. The multiplication morphism for $\bG$ will be denoted by $\bm$.

Let $\Frob{}$ denote the geometric Frobenius element in $\Gal(\bFq/\Fq)$ as
well as the corresponding automorphism of $\Spec{\bFq}$. The Weil group
$\Weil{}\subset \Gal(\bFq/\Fq)$ is the subgroup generated by $\Frob{}$.
Let $\Frob{G} \ceq \id_{G} \times \Frob{}$ be the Frobenius automorphism of $\bG = G \times_{\Spec{\Fq}} \Spec{\bFq}$.

We fix a prime $\ell$, invertible in $\Fq$.
We will work with constructible $\ell$-adic sheaves \citelist{\cite{deligne:80a}*{\S 1.1} \cite{SGA5}*{Expos\'es V, VI}}
on schemes locally of finite type over $\Fq$, employing the standard formalism.
We also make extensive use of the external tensor product of $\ell$-adic sheaves,
defined as follows: if $\mathcal{F}$ and $\mathcal{G}$ are constructible $\ell$-adic
sheaves on schemes $X$ and $Y$ and $p_X : X\times Y\to X$ and $p_Y : X\times Y \to Y$
are the projections, then $\mathcal{F}\boxtimes \mathcal{G} \ceq p_X^* \mathcal{F} \otimes p_Y^*\mathcal{G}$.

For any commutative group $A$, we will write $A^*$ for the dual group $\Hom(A, \EEx)$.

\subsection{Character sheaves on commutative group schemes over finite fields}\label{ssec:category}

\begin{definition}\label{def:CS}
A \emph{character sheaf on $G$} is a triple
$\cs{L}\ceq (\gcs{L},\mu,\phi)$ where:
\begin{enumerate}
\labitem{(CS.1)}{CS.1} $\gcs{L}$ is a rank-one $\ell$-adic local system on $\bG$, by which we mean a constructible
$\ell$-adic sheaf on $\bG$, {\it lisse} on each connected component of $\bG$, whose stalks are one-dimensional $\EE$-vector spaces;
\labitem{(CS.2)}{CS.2} $\mu: \bm^* \gcs{L} \to \LxL$ is an isomorphism of
sheaves on $\bGxG$ such that the following diagram commutes,
  where $m_3 \ceq m\circ (m\tight{1}{\times}{2}\id) = m\circ (\id\tight{2}{\times}{1} m)$;
  \[
  \begin{tikzcd}[row sep=30]
  \bm_3^*\gcs{L} \arrow{rr}{(\bm \tight{1}{\times}{2} \id)^*\mu} \arrow[swap]{d}{(\id \tight{2}{\times}{1} \bm)^*\mu}
    &&  \bm^*\gcs{L} \boxtimes \gcs{L} \dar{\mu \tight{0}{\boxtimes}{1} \id} \\
    \gcs{L} \boxtimes \bm^* \gcs{L} \arrow{rr}{\id \boxtimes \mu}
    &&  \gcs{L} \tight{0}{\boxtimes}{0} \LxL
  \end{tikzcd}
  \]
\labitem{(CS.3)}{CS.3} $\phi : \Frob{G}^* \gcs{L} \to \gcs{L}$ is an
  isomorphism of constructible $\ell$-adic sheaves on $\bG$ compatible with
  $\mu$ in the sense that the following diagram commutes.
  \[
  \begin{tikzcd}[row sep=20]
  \Frob{\GxxG}^* \bm^* \gcs{L} \arrow{rr}{\Frob{\GxxG}^*\mu}
    && \Frob{\GxxG}^*(\LxL)\\
    \arrow[equal]{u} \bm^*  \Frob{G}^* \gcs{L} \arrow[swap]{d}{\bm^* \phi}
    && \Frob{G}^*\gcs{L}\boxtimes \Frob{G}^*\gcs{L} \dar{\phi\boxtimes \phi} \arrow[equal]{u} \\
    \bm^*\gcs{L} \arrow{rr}{\mu}
    && \LxL
  \end{tikzcd}
  \]
\end{enumerate}
Morphisms of character sheaves are defined in the natural way:
\begin{enumerate}
\labitem{(CS.4)}{CS.4} if $\cs{L} = (\gcs{L},\mu,\phi)$ and
  $\cs{L'} = (\gcs{L'},\mu',\phi')$ are character sheaves on $G$ then
  a morphism $\rho : \cs{L} \to \cs{L}'$ is a map $\brho : \gcs{L} \to \gcs{L'}$
  of constructible $\ell$-adic sheaves on $\bG$ such that the following diagrams both commute.
  \[
  \begin{tikzcd}[column sep=40]
  \Frob{G}^* \gcs{L} \rar{\Frob{G}^* \brho} \arrow[swap]{d}{\phi} & \Frob{G}^* \gcs{L'} \dar{\phi'}
  & & \arrow[swap]{d}{\mu} \bm^* \gcs{L} \rar{\bm^* \brho} & \bm^* \gcs{L'} \dar{\mu'} \\
  \gcs{L} \rar{\brho} & \gcs{L'}
  & {} & \LxL \rar{\tight{1}{\rho\boxtimes \rho}{1}} & \gcs{L'} \tight{0}{\boxtimes}{0} \gcs{L'}
  \end{tikzcd}
  \]
\end{enumerate}
 The category of character sheaves on $G$ will be denoted by $\CS(G)$.
 \end{definition}

Category $\CS(G)$ is a rigid monoidal category
\cite{etingof:09a}*{\S1.10} under the tensor product
$\cs{L} \otimes \cs{L'}$ defined by $(\gcs{L}\otimes\gcs{L'}, \mu\otimes\mu', \phi\otimes \phi')$,
with duals given by applying the sheaf hom functor
$\sheafHom(\ - \ ,\EE)$.
This rigid monoidal category structure for $\CS(G)$ gives the set $\CSiso{G}$
of isomorphism classes in $\CS(G)$ the structure of a group.

\begin{remark}
The category of character sheaves on $G$ is not abelian since it is not closed under direct sums; 
thus $\CS(G)$ is not a tensor category in the sense of \cite{deligne:02a}*{0.1}.  
We suspect that requiring that $\mu$ be injective rather than
an isomorphism and dropping the condition that the stalks be one-dimensional would yield an abelian category.
\end{remark}

We will describe the group $\CSiso{G}$ in Theorem~\ref{thm:snake}
and the sets $\Hom(\cs{L},\cs{L}')$ in Theorem~\ref{thm:autornaught}; in this way we provide a complete description of the category $\CS(G)$.
In the meantime, we make an elementary observation about $\Hom(\cs{L},\cs{L}')$.

\begin{lemma}\label{lem:autornaught}
Let $G$ be a smooth commutative group scheme over $\Fq$.
If $\cs{L}$ and $\cs{L}'$ are character sheaves on $G$, then
every $\rho\in \Hom(\cs{L},\cs{L}')$ is either trivial (zero on every stalk)
or an isomorphism. 
\end{lemma}

\begin{proof}
Suppose $\rho \in \Hom(\cs{L},\cs{L}')$.
We prove the lemma by considering the linear transformations $\brho_{\bg} : \gcs{L}_{\bg} \to \gcs{L}_{\bg}$ at the stalks
above geometric points $\bg$ on $G$ and showing that, either each $\brho_{\bg}$ trivial or each $\brho_{\bg}$ is an isomorphism.
(This idea is expanded upon in Section~\ref{ssec:stalks}.)
Let ${\bar e}$ be the geometric point above the identity $e$ for $G$ determined by our choice of algebraic closure $\bFq$ of $\Fq$.
If $\brho_{\bar e} = 0$ then the second diagram in \ref{CS.4} implies that $\brho_{\bg} = 0$ for all $\bg$, in which case $\rho$ is trivial.
On the other hand, if $\brho_{\bar e}$ is non-trivial then the second diagram in \ref{CS.4} implies that $\brho_{\bg}$ is non-trivial
for all $\bg$ and thus an isomorphism, since the stalks of character sheaves are one-dimensional; in this case $\rho$ is an isomorphism.
\end{proof}

\subsection{Trace of Frobenius}\label{ssec:Frob}

In this section we introduce two tools which will help us understand isomorphism classes of objects in $\CS(G)$:
the map $\CSiso{G} \to G(k)^*$ given by trace of Frobenius and the pullback functor $\CS(G) \to \CS(H)$
associated to a morphism $H \to G$ of smooth group schemes over $\Fq$.

Let $(\gcs{L},\phi)$ be a Weil sheaf on $G$. Every $g\in G(\Fq)$
determines a geometric point $\bg$ fixed by $\Frob{G}$. 
Together with the canonical isomorphism $(\Frob{G}^*\gcs{L})_{\bg} \iso  \gcs{L}_{\Frob{G}(\bg)}$,
the automorphism $\phi$ determines an automorphism $\phi_{\bg}$ of the $\EE$-vector space $\gcs{L}_{\bg}$.
Let $\trace(\phi_{\bg};\gcs{L}_{\bg})$ be the trace of $\phi_{\bg} \in \Aut_{\EE}(\gcs{L}_{\bg})$ and let
$\trFrob{(\gcs{L},\phi)} : G(\Fq)\to \EE$ be the function defined by 
\begin{equation}\label{trWeil}
\trFrob{(\gcs{L},\phi)}(g) \ceq \trace(\phi_{\bg};\gcs{L}_{\bg}),
\end{equation}
commonly called the {\em trace of Frobenius of $(\gcs{L},\phi)$}.
Note that if $(\gcs{L},\phi) \iso (\gcs{L'},\phi')$ as Weil sheaves, 
then $\trFrob{(\gcs{L},\phi)} = \trFrob{(\gcs{L'},\phi')}$ as functions on $G(\Fq)$.

Now suppose $\cs{L} = (\gcs{L},\mu,\phi)$ is a character sheaf on $G$.
Then the isomorphism $\bm^* \gcs{L} \iso \gcs{L} \boxtimes\gcs{L}$ and the diagram of
\ref{CS.3} guarantee
that the function $\trFrob{(\gcs{L},\phi)} : G(\Fq)\to \EEx$ is a group homomorphism, which we will also denote by $\trFrob{\cs{L}}$.
Moreover, this homomorphism depends only on the isomorphism class of $\cs{L}$, so we obtain a map
\begin{align*}
\TrFrob{G} : \CSiso{G} &\to G(\Fq)^*, \\
\cs{L} &\mapsto \trFrob{\cs{L}}.
\end{align*}
Since tensor products on the stalks of $\cs{L}$ induce pointwise multiplication on the trace of Frobenius, $\TrFrob{G}$ is a group homomorphism.  

The next two results follow easily from the definitions.

\begin{lemma}\label{lem:pullback}
  If $f : H\to G$ is a morphism of smooth commutative group schemes over $\Fq$, then
  \begin{align*}
  f^* : \CS(G) &\to \CS(H) \\
  (\gcs{L},\mu,\phi) &\mapsto (\bf^*\gcs{L},(\bfxf)^*\mu,\bf^*F)
  \end{align*}
  defines a monoidal functor dual to $f \colon H(\Fq) \to G(\Fq)$ in the sense that
  \[
  \begin{tikzcd}[row sep=20, column sep=30]
   \CSiso{G} \rar{f^*} \arrow[swap]{d}{\TrFrob{G}} & \CSiso{H} \dar{\TrFrob{H}} \\
   G(\Fq)^* \rar & H(\Fq)^*
  \end{tikzcd}
  \]
  is a commutative diagram of groups.  Moreover, $(f\circ g)^* = g^* \circ f^*$.
\end{lemma}

\begin{proof}
  Let $\cs{L}$ be a character sheaf on $G$. 
  Pullback by $\bf$ takes rank-one local systems to rank-one local systems.
  To see that $(\bfxf)^* \mu$ satisfies \ref{CS.2},
  apply the functor $(\bfxf)^*$
  to \ref{CS.2} for $\cs{L}$ and use the canonical isomorphism
  $(\bfxf)^*(\LxL) \iso \bf^*\gcs{L} \tight{-3}{\boxtimes}{-3} \bf^*\gcs{L}$.
  To show that $f^*\cs{L}$ satisfies
  \ref{CS.3}, apply the same functor to \ref{CS.3} for $\cs{L}$.
  Since $f$ is a morphism of group schemes defined over $\Fq$
  it provides isomorphisms $(\bfxf)^*\Frob{\GxxG}^* \iso \Frob{\GxxG}^* (\bfxf)^*$
  and $(\bfxf)^* \bm^*\iso \bm^* \bf^*$ between functors of constructible sheaves.

  Applying $\bf^*$ and $\bf^* \tight{1}{\times}{1}\bf^*$ to \ref{CS.4} defines the action
  of $f^*$ on morphisms of character sheaves; arguing as above shows that $f^*$ is
  a functor from $\CS(G)$ to $\CS(H)$.  Since tensor products commute with pullback in schemes,
  $f^* : \CS(G) \to \CS(H)$ is a monoidal functor.
  The diagram relating $f^* : \CS(G) \to \CS(H)$, $f^* : G(k)^* \to H(k)^*$ and trace of Frobenius
  commutes by \cite{laumon:87a}*{1.1.1.2}, where the ambient
 finite type hypothesis can be replaced by locally of finite type.

  Finally, the fact that $(f\circ g)^* = g^* \circ f^*$ follows from the analogous
  statements about the pullback functor on $\ell$-adic constructible sheaves.
\end{proof}

If $G_1$ and $G_2$ are smooth commutative group schemes over $\Fq$ then characters of $(G_1 \times G_2)(\Fq)$
all take the form $\chi_1\otimes \chi_2$ for characters $\chi_1$ of $G_1(\Fq)$ and $\chi_2$ of $G_2(\Fq)$. 
The next lemma shows that character sheaves on $G$ enjoy an analogous property.

\begin{lemma}\label{lem:product}
If $G_1$ and $G_2$ are smooth commutative group schemes over $\Fq$ then the following diagram commutes.
\[
\begin{tikzcd}[column sep=60]
\arrow{d}{\TrFrob{G_1} \times \TrFrob{G_2}} \CSiso{G_1}\times \CSiso{G_2} \arrow{r}{(\cs{L}_1,\cs{L}_2)\mapsto \cs{L}_1\boxtimes \cs{L}_2}
& \arrow{d}{\TrFrob{G_1\times G_2}} \CSiso{G_1\times G_2}\\
(G_1)(\Fq)^*\times (G_2)(\Fq)^* \arrow{r}{(\chi_1,\chi_2)\mapsto \chi_1\otimes \chi_2}  & (G_1\times G_2)(\Fq)^*
\end{tikzcd}
\]
Moreover, every character sheaf on $G_1\times G_2$ is isomorphic to $\cs{L}_1\boxtimes\cs{L}_2$
for some character sheaves $\cs{L}_1$ on $G_1$ and $\cs{L}_2$ on $G_2$.
\end{lemma}

\begin{proof}
The only non-trivial part is the last claim, so we will only address that point here.
Set $G \ceq G_1\times G_2$
and write $e_1$ and $e_2$ for the identity elements of $G_1$ and $G_2$.
Define $f : G\to G\times G$ by $f(g_1,g_2) \ceq (g_1,e_2,e_1,g_2)$.
Observe that $m\circ f = \id_G$.
Let $p_1$, $p_2$ be the projection morphisms pictured below:
\[
\begin{tikzcd}
G & \arrow[swap]{l}{p_1} G\times G \arrow{r}{p_2} & G.
\end{tikzcd}
\]
Let $r_1$ and $r_2$ be the projection morphisms pictured below,
with sections $q_1$ and $q_2$, also morphisms of group schemes:
\[
\begin{tikzcd}
G_1  \arrow[swap, bend right]{r}{q_1} &
\arrow[swap, bend right]{l}{r_1} G_1\times G_2 \arrow[bend left]{r}{r_2} &
\arrow[bend left]{l}{q_2} G_2.
\end{tikzcd}
\]
Observe that $p_1\circ f = q_1 \circ r_1$ and $p_2 \circ f = q_2\circ r_2$.
Now, let $\cs{L} \ceq (\gcs{L},\mu,\phi)$ be a character sheaf on $G$
and set $\cs{L}_1 \ceq q_1^* \cs{L}$ and $\cs{L}_2 \ceq q_2^* \cs{L}$.
By Lemma~\ref{lem:pullback}, $\cs{L}_1$ is a character sheaf on $G_1$
and $\cs{L}_2$ is a character sheaf on $G_2$.
We will obtain an isomorphism $\cs{L} \iso  \cs{L}_1\boxtimes \cs{L}_2$.

Applying the functor $f^*$ to the isomorphism $\mu$ yields
\begin{equation}\label{eq:fm}
f^*\mu : f^* m^* \gcs{L} \to f^*(\gcs{L}\boxtimes \gcs{L}) .
\end{equation}
We have already seen that $m\circ f = \id_G$, so $f^* m^* \gcs{L} = \gcs{L}$.  
Since $f^*p_1^*\gcs{L} = r_1^* q_1^* \gcs{L} = r_1^* \gcs{L}_1$ and $f^*p_2^*\gcs{L} = r_2^* q_2^* \gcs{L} = r_2^*\gcs{L}_2$,
we have 
\[
f^*(\gcs{L}\boxtimes \gcs{L})  = f^*p_1^*\gcs{L}\otimes f^* p_2^*\gcs{L} = \gcs{L}_1\boxtimes \gcs{L}_2.
\]
It follows that \eqref{eq:fm} gives an isomorphism $\gcs{L} \to  \gcs{L}_1\boxtimes \gcs{L}_2$.
It is routine to show that this morphism satisfies
\ref{CS.4}, as it applies here,
from which it follows that we have exhibited an isomorphism
$\cs{L} \to \cs{L}_1\boxtimes \cs{L}_2$ of characters sheaves on $G\times G$.
\end{proof}

Using these results on pullbacks and products, we may prove a naturality property of $\TrFrob{G}$.

\begin{proposition}\label{prop:functorialG}
The homomorphism $\TrFrob{G} : \CSiso{G} \to G(\Fq)^*$ defines a natural transformation
between the two contravariant additive functors
\begin{align*}
F_1 : G &\mapsto \CSiso{G} \\
F_2 : G &\mapsto G(\Fq)^*
\end{align*}
from the category of smooth commutative group schemes over $\Fq$ to the category of commutative groups.
\end{proposition}

\begin{proof}
The first part of Lemma~\ref{lem:pullback} shows that $F_1$ is a functor,
while the second part shows that Trace of Frobenius is a natural transformation
$T: F_1 \to F_2$. When further combined with Lemma~\ref{lem:product},
we see that $F_1$ is an additive functor and $T: F_1 \to F_2$ is a natural
transformation between additive functors,
concluding the proof of Proposition~\ref{prop:functorialG}.
\end{proof}

\subsection{Descent}\label{ssec:descentG}

In this section we consider a category of sheaves on $G$ obtained by replacing the Weil sheaf $(\gcs{L}, \phi)$
on $\bG$ in the definition of a character sheaf with an $\ell$-adic local system on $G$ itself;
these will play a role in Sections~\ref{ssec:connected} and \ref{ssec:alg_groups}.

\begin{definition}
Let $\bCS(G)$ be the category of pairs $(\cs{E},\mu_\cs{E})$
where $\cs{E}$ is an $\ell$-adic local system on $G$ of rank-one,
equipped with an isomorphism $\mu_\cs{E} : m^* \cs{E} \to \cs{E} \boxtimes \cs{E}$
satisfying the analogue of \ref{CS.2} on $G$;
morphisms in $\bCS(G)$ are defined as in the second part of
\ref{CS.4}. 
\end{definition}

We put a rigid monoidal structure on $\bCS(G)$ in the same way as for $\CS(G)$.

\begin{proposition}\label{prop:BG}
Extension of scalars defines a full and faithful functor
\[
B_G : \bCS(G) \to \CS(G).
\]
\end{proposition}

\begin{proof}
 Suppose $(\cs{E},\mu_\cs{E})$ in an object of $\bCS(G)$.
 Let $b_G : \bG \to G$ be the pullback of $\Spec{\bFq} \to \Spec{\Fq}$ along $G\to \Spec{\Fq}$.
 Set $\gcs{L} = b_G^* \cs{E}$. 
 The functor $b_G^*$ takes local systems on $G$ to local systems on $\bG$.
 The local system $\gcs{L}$ comes equipped with an isomorphism
 $\phi: \Frob{G}^* \gcs{L} \to \gcs{L}$.
 The resulting functor from local systems on $G$ to Weil local systems on $\bG$, given on objects by $\cs{E} \mapsto (\gcs{L},\phi)$, 
 is full and faithful; see \cite{deligne-katz:SGA7.2}*{Expos\'e XIII} and \cite{beilinson-bernstein-deligne:81a}*{Prop. 5.1.2}.
 The isomorphism $\mu \ceq b_{G\times G}^*\mu_\cs{E}$ satisfies \ref{CS.2}
 for $\gcs{L}$ and $\phi$ is compatible with $\mu$ in the sense of \ref{CS.3}.
 This construction defines the functor $B_G : \bCS(G) \to \CS(G)$ given on objects by $(\cs{E},\mu_\cs{E}) \mapsto (\gcs{L},\mu, \phi)$, as defined here. 
 Because morphisms in $\bCS(G)$ and $\CS(G)$ are morphisms of local systems on $G$ and $\bG$,
 respectively, satisfying condition~\ref{CS.4}, this functor is also full and faithful.
\end{proof}

We will say that a character sheaf $\cs{L} \in \CS(G)$ \emph{descends to $G$} if it is isomorphic to some $B_G(\cs{E}, \mu_\cs{E})$.

\begin{remark}\label{rem:descent}
In fact, it is not difficult to recognize character sheaves that descend to $G$: they are exactly those character sheaves
$\cs{L} = (\gcs{L},\mu,\phi)$ for which the action of $W$ on $\gcs{L}$ given by $\phi$ extends to a continuous action
of $\Gal(\bFq/\Fq)$ on $\gcs{L}$; see \cite{deligne-katz:SGA7.2}*{Expos\'e XIII, Rappel 1.1.3} for example. 
\end{remark}

\subsection{Discrete isogenies}\label{ssec:discrete_isogenies}

Here we consider character sheaves on $G$ that are defined by discrete isogenies onto $G$
(\S \ref{ssec:discrete_isogenies}); these will play a role in Section~\ref{ssec:restriction}.

A finite, \'etale, surjective morphism $H\to G$ of smooth group schemes over $\Fq$ for which the action
of $\Gal(\bFq/\Fq)$ on the kernel is trivial is called a {\it discrete isogeny}, inspired by \cite{kamgarpour:09a}*{\S 2.2}.

\begin{proposition}\label{prop:finite}
Let $f: H \to G$ be a discrete isogeny and let $A$ be the kernel of $f$.
Let $V$ be a $1$-dimensional representation of $A$ 
equipped with an isomorphism $V\to V\otimes V$.
Let $\psi : A \to \EEx$ be the character of $V$.
Then $(f_! V_H)_\psi$ (the $\psi$-isotypic component of $f_!V_H$) is an object of $\bCS(G)$.
\end{proposition}

\begin{proof}
Let $f$, $A$, $V$ and $\psi$ be as above and set $\cs{E} = (f_! V_H)_\psi$.
Since $A$ is abelian, $\cs{E}$ is an $\ell$-adic local system on $G$ of rank one.
We must show that $\cs{E}$ comes equipped with an isomorphism $\mu_\cs{E} : m^* \cs{E} \to \cs{E}\boxtimes\cs{E}$.
To do this we use \'etale descent to see that pullback along $f$ gives an equivalence between $\ell$-adic local systems
on $G$ and $A$-equivariant local systems on $H$ {\it cf.\ } \cite{bernstein-luntz:equivariant_sheaves}*{Prop 8.1.1}. 
In particular, $f^*\cs{E}$ is the $A$-equivariant constant sheaf $V$ on $H$ with character $\psi$.
Since $f$ is a morphism of group schemes, the functor $f^*$ defines $\mu_\cs{E} : m^*\cs{E} \to \cs{E}\boxtimes\cs{E}$
from the isomorphism $m^*\psi \iso \psi \boxtimes\psi$ determined by $V\to V\otimes V$. 
\end{proof}

\begin{remark}
Since $V$ is $1$-dimensional, the choice of $V \to V\otimes V$ is exactly the choice of an isomorphism $V\iso \EE$.
\end{remark}

\begin{remark}
A descent argument similar to the one employed in the proof of Lemma~\ref{prop:finite} is used in
\cite{boyarchenko-drinfeld:10a}*{Lemma~1.10}, though in the more restrictive case of connected algebraic groups.
\end{remark}

\subsection{Recollections on character sheaves for connected algebraic groups}\label{ssec:connected}

If the smooth commutative group scheme $G$ is of finite type, then every character sheaf descends to $G$;
we will see that they are not equivalent, in general, however.

\begin{lemma}\label{lem:bounded_connected}
If $G$ is a connected commutative algebraic group over $\Fq$ then 
\[
B_G : \bCS(G) \to \CS(G)
\]
 is an equivalence of categories.
\end{lemma}

\begin{proof}
Choose any $\Fq$-rational point $g$ on $G$ and let $\bg$ be the geometric point on $G$ lying above $g$.
Recall that the \emph{Weil group} of $G$, which we will denote by $\Weil{}(G,\bg)$, is a subgroup of the \'etale
fundamental group defined by the following diagram:
\[
 \begin{tikzcd}
 1 \rar & \ar[equal]{d} \pi_1(\bG, \bg) \rar & \Weil{}(G,\bg) \rar \dar[hook] & \Weil{} \rar \dar[hook] & 1 \\
 1 \rar &  \pi_1(\bG, \bg) \rar & \pi_1(G,\bg) \rar & \Gal(\bFq/\Fq) \rar & 1.
 \end{tikzcd}
\]
The $\Fq$-rational point $g$ under the geometric point $\bg$ determines a splitting
$\Weil{}\to \Weil{}(G,\bg)$ of $\Weil{}(G,\bg)\to \Weil{}$.
  Since $G$ is connected, the geometric point $\bg$ determines
  an equivalence between the category of $\ell$-adic Weil local systems on $G$ and
  $\ell$-adic representations of $\Weil{}(G,\bg)$ \cite{deligne:80a}*{1.1.12}.
  
  Now let $(\gcs{L},\mu,\phi)$ be a character sheaf on $G$
  and let $\lambda : \Weil{}(G, \bg) \to \EEx$ be the character determined by $(\gcs{L},\phi)$.
  Composing with the splitting $\Weil{} \to \Weil{}(G,\bg)$ yields an $\ell$-adic character
  $\lambda_g : \Weil{} \to \EEx$, which is the same as the Trace of Frobenius defined in Section~\ref{ssec:Frob}, for every $\Fq$ rational point $g$ on $G$:
  $
  \lambda_g(\Frob{}) =  \trFrob{\cs{L}}(g).
  $
  On the other hand, we have already seen that $\trFrob{\cs{L}} : G(\Fq) \to \EEx$
  is a group homomorphism. 
  Since $G$ is an algebraic group over $\Fq$, $G(\Fq)$ is finite.
  Therefore $\trFrob{\cs{L}}(g) = \lambda_g(\Frob{})$ is a root of unity
  for every $g\in G(\Fq)$.  Since $\Weil{}$ is generated by
  $\Frob{}$ and $\lambda_g : \Weil{} \to \EEx$ is
  a character, it follows that the image of $\lambda_g$ is a finite group.
  Thus, $\lambda_g$ extends to an $\ell$-adic character of $\Gal(\bFq/\Fq)$,
  which we will also denote $\lambda_g$.
  We may now lift the $\ell$-adic character $\lambda_g : \Gal(\bFq/\Fq) \to \EEx$
  to an $\ell$-adic character $\pi_1(G,\bg) \to \EEx$ using the canonical topological group homomorphism
  $\pi_1(G,\bg) \to \Gal(\bFq/\Fq)$. 
  The $\Fq$ rational point $g$ also
  determines an equivalence between the category of $\ell$-adic
  representations of $\pi_1(G,\bg)$ and $\ell$-adic local systems on $G$. Let
  $\cs{E}$ be a local system on $G$ in the isomorphism class
  determined by this $\ell$-adic character of $\pi_1(G,\bg)$.
  Then $b_G^*\cs{E} \iso \gcs{L}$.
  
  Since $b_{G\times G}^*$ is full and faithful (again, see
\cite{deligne-katz:SGA7.2}*{Expos\'e XIII} or \cite{beilinson-bernstein-deligne:81a}*{Prop. 5.1.2}),
 \[
  b_{G\times G}^* : \Hom(m^*\cs{E},\cs{E}\boxtimes\cs{E}) \to \Hom(\bm^*\gcs{L},\gcs{L}\boxtimes\gcs{L})
 \]
  is a bijection
  (hom taken in the categories on constructible $\ell$-adic sheaves on
  $G\times G$ and $\bG\times \bG$ respectively,
  in which $\ell$-adic local systems sit as full subcategories).
  Let $\mu_\cs{E} : m^*\cs{E} \to \cs{E}\boxtimes\cs{E}$ be the isomorphism matching
  $\mu : \bm^*\gcs{L} \to \gcs{L}\boxtimes\gcs{L}$,
  the latter appearing in the definition of $\cs{L}$.
  Then, as in Section~\ref{ssec:descentG}, $(\cs{E},\mu_\cs{E})$ is an object of $\bCS(G)$
  and $\cs{L} \ceq (\gcs{L},\mu,\phi)$ is isomorphic to $(b_G^*\cs{E},b_{G\times G}^*\mu_\cs{E})$ in $\CS(G)$.
  Thus, the full and faithful functor $B_G : \bCS(G) \to \CS(G)$ from Section~\ref{ssec:descentG}
  is also essentially surjective, hence an equivalence.
\end{proof}

Using this equivalence of categories, we may give a good description of $\CS(G)$ when $G$ is connected and finite type.

\begin{proposition}\label{prop:connected}
 If $G$ is a connected, commutative algebraic group over $\Fq$ then:
 \begin{enumerate}
 \labitem{(1)}{c1} $\TrFrob{G} : \CSiso{G} \to G(\Fq)^*$ is an isomorphism of groups;
 \labitem{(2)}{c2} every character sheaf on $G$ is isomorphic to one defined by a discrete isogeny;
 \labitem{(3)}{c3} $\Aut(\cs{L}) = 1$, for all character sheaves $\cs{L}$ on $G$.
 \end{enumerate}
 \end{proposition}
\begin{proof}
By Lemma~\ref{lem:bounded_connected}, we know that every character sheaf $\cs{L}$ on $\bG$ descends to $G$;
let $\cs{E}$ be an object of $\bCS(G)$ for which $B_G(\cs{E}) \iso \cs{L}$.
 Since the functor $B_G : \bCS(G) \to \CS(G)$ is full and faithful, $\Aut(\cs{L}) = \Aut(\cs{E})$.
From here, Deligne's function-sheaf dictionary for connected commutative algebraic groups over finite fields,
as in \cite{deligne:SGA4.5}*{Sommes trig.} or \cite{laumon:87a}*{1.1.3}, gives us all we need for points \ref{c1} and \ref{c2}, as we briefly recall.

As in the proof of Proposition~\ref{prop:finite}, use \'etale descent to see that pullback by the Lang isogeny $\Lang : G\to G$
defines an equivalence of categories between local systems on $G$ and $G(\Fq)$-equivariant local systems on $G$. 
Under this equivalence, local systems $\cs{E}$ on $G$ arising from objects in $\bCS(G)$ are matched with $G(\Fq)$-equivariant
constant local systems of rank-one on $G$, and therefore with one-dimensional representations of $G(\Fq)$. 
In the same way, pullback along the isogeny $\Lang\times\Lang : G\times G\to G\times G$ matches the extra structure
$\mu_\cs{E} : m^*\cs{E} \to \cs{E}\boxtimes\cs{E}$ with an isomorphism $m^*V \to V\boxtimes V$ of one-dimensional
representations of $G(\Fq)\times G(\Fq)$, which is exactly an isomorphism $V \to V\otimes V$ of one-dimensional representations,
which is exactly the choice of an isomorphism $V\iso \EE$.
We see that $\bCS(G)$ is equivalent to the category of characters of $G(\Fq)$.
Let $\bg$ be a geometric point above $g \in G(\Fq)$.  If $\cs{E}$ matches $\psi : G(\Fq)\to \EEx$ under this equivalence,
a simple calculation on stalks reveals that the action of Frobenius on $\cs{E}_{\bg}$ is multiplication by $\psi(g)^{-1}$.
In other words, for every $\cs{E}$ in $\bCS(G)$, the trace of $\Lang^*\cs{E}$ is $\trFrob{\cs{E}}^{-1}$ as a
representation of $G(\Fq)$, proving parts \ref{c1} and \ref{c2}.

For part \ref{c3}, suppose $\Lang^*\cs{E} = V$ with isomorphism $V \to V\otimes V$.
Observe that the equivalence above establishes a bijection between $\Aut(\cs{E})$ and the group of automorphisms of $\rho : V\to V$ for which 
\[
\begin{tikzcd}
\arrow{d}{} V \arrow{r}{\rho} & V\arrow{d}{}\\
V\otimes V \arrow{r}{\rho\otimes \rho} & V\otimes V
\end{tikzcd}
\]
commutes. 
Since the only such isomorphism $\rho$ is $\id_V$, it follows that $\Aut(\cs{E}) = 1$, completing the proof.
\end{proof}

We have just seen that, for a connected commutative algebraic group $G$ over $\Fq$, the category of character sheaves
on $G$ is equivalent to the category of one-dimensional representations $V$ of $G(\Fq)$ equipped with an isomorphism
$V\iso \EE$, and therefore equivalent to the category of characters $\psi$ of $G(\Fq)$.
We have also just seen that if the character of $\Lang^*\cs{E}$ is $\psi$ then the canonical isomorphism
$m^*\psi \iso \psi \boxtimes \psi$ determines the isomorphism $\mu_\cs{E} : \cs{E} \to \cs{E}\boxtimes\cs{E}$.
This fact leads (back) to a perspective on the function-sheaf dictionary common in the literature in which one considers
one-dimensional local systems $\cs{E}$ on $G$ for which \emph{there exists} an isomorphism
$m^*\cs{E} \iso \cs{E} \boxtimes\cs{E}$ \cite{kamgarpour:09a}*{Intro}.
As a slight variation, one may also consider one-dimensional local systems $\gcs{L}$ on $\bG$ for which \emph{there exists}
an isomorphism $\Frob{G}^*\gcs{L} \iso \gcs{L}$ and an isomorphism $\bm^*\gcs{L} \iso \gcs{L} \boxtimes\gcs{L}$.

Although the category $\CS(G)$ of character sheaves on $G$ specializes to $\bCS(G)$ when $G$ is of finite type
(\S \ref{ssec:alg_groups}), this description is \emph{not} sufficient when extending the dictionary to smooth
commutative group schemes, as we will see already in Section~\ref{sec:etale}.
In particular, for a given $\gcs{L}$ and $\phi$ there may be many $\mu$ so that $(\gcs{L},\mu,\phi)$ is a character sheaf.
For \'etale $G$, Proposition~\ref{prop:etale-iso} shows that $\Hh^2(\bG,\EEx)^{\Weil{}}$ measures the possibilities for $\mu$.
We will see in Section~\ref{sec:main} that $\Hh^2(\pi_0(\bG),\EEx)^{\Weil{}}$ plays an analogous role for general
smooth commutative group schemes $G$.

\section{Character sheaves on \'etale commutative group schemes over finite fields} \label{sec:etale}

In this section we give a complete characterization of the category of character sheaves on \'etale commutative group schemes over finite fields.

\subsection{Stalks of character sheaves}\label{ssec:stalks}

The equivalence $G \mapsto G(\bFq)$,
from the category of \'etale commutative group schemes over $\Fq$ to the category of commutative groups equipped
with a continuous action of $\Gal(\bFq/\Fq)$,
provides the following simple description of character sheaves.
A character sheaf $\cs{L}$ on an \'etale commutative group scheme $G$ over $\Fq$ is:
\begin{enumerate}
 \labitem{(cs.1)}{cs.1} an indexed set of one-dimensional
  $\EE$-vector spaces $\gcs{L}_x$, as $x$ runs over
  $G(\bFq)$;

 \labitem{(cs.2)}{cs.2} an indexed set of isomorphisms
  $\mu_{x,y} : \gcs{L}_{x+y} \xrightarrow{\iso} \gcs{L}_{x} \otimes\gcs{L}_{y}$,
  for all $x,y \in G(\bFq)$, such that
  \[
   \begin{tikzcd}[row sep=40]
    \gcs{L}_{x+y+z} \arrow{rr}{\mu_{x+y,z}} \arrow[swap]{d}{\mu_{x,y+z}}
    && \gcs{L}_{x+y}\otimes\gcs{L}_{z} \dar{\mu_{x,y} \tight{0.5}{\otimes}{1} \id} \\
    \gcs{L}_{x} \otimes\gcs{L}_{y+z} \arrow{rr}{\id \otimes\mu_{y,z}}
    && \gcs{L}_{x} \otimes\gcs{L}_{y} \otimes\gcs{L}_{z}
   \end{tikzcd}
  \]
  commutes, for all $x,y,z\in G(\bFq)$; and
 \labitem{(cs.3)}{cs.3} an indexed set of isomorphisms $\phi_{x} : \gcs{L}_{\Frob{}(x)} \to \gcs{L}_x$
  such that
  \[
   \begin{tikzcd}[row sep=40]
    \gcs{L}_{\Frob{}(x)+\Frob{}(y)} \arrow[swap]{d}{\phi_{x+y}} \arrow{rr}{\mu_{\Frob{}(x),\Frob{}(y)}}
    && \gcs{L}_{\Frob{}(x)}\otimes\gcs{L}_{\Frob{}(y)} \dar{\phi_x \tight{0}{\otimes}{0} \phi_y} \\
    \gcs{L}_{x+y} \arrow{rr}{\mu_{x,y}}
    && \gcs{L}_x \otimes\gcs{L}_y
   \end{tikzcd}
  \]
  commutes, for all $x,y\in G(\bFq)$.
\end{enumerate}
Under this equivalence, a morphism $\rho : \cs{L} \to \cs{L'}$ of character sheaves on $G$ is given by
\begin{enumerate}
 \labitem{(cs.4)}{cs.4} an indexed set $\brho_x : \gcs{L}_x \to \gcs{L'}_x$
  of linear transformations such that
  \[
   \begin{tikzcd}[column sep=40]
    \arrow[swap]{d}{\phi_x} \gcs{L}_{\Frob{}(x)} \rar{\brho_{\Frob{}(x)}} & \gcs{L'}_{\Frob{}(x)} \dar{\phi_x'}
    &\arrow[draw=none]{d}[pos=.4,description]{\text{\normalsize{and}}}
    & \arrow[swap]{d}{\mu_{x,y}} \gcs{L}_{x+y} \rar{\brho_{x+y}} & \gcs{L'}_{x+y} \dar{\mu'_{x,y}} \\
    \gcs{L}_x \rar{\brho_x} & \gcs{L'}_x
    & {} & \gcs{L}_x\otimes\gcs{L}_y \rar{\brho_x\otimes\brho_y} & \gcs{L'}_x \otimes\gcs{L'}_y
   \end{tikzcd}
  \]
  both commute, for all $x, y \in G(\bFq)$.
\end{enumerate}

We will see that $\TrFrob{G} : \CSiso{G} \to G(\Fq)^*$ may not provide complete
information about isomorphism classes of character sheaves on $G$ when $G$ is not a connected algebraic group.
Our main tool for understanding this phenomenon
is a group homomorphism $S_G: \CSiso{G} \to \Hh^2(E^\bullet_G)$ defined in Section~\ref{ssec:S}, for which the next two sections are preparation.

\subsection{A spectral sequence}\label{ssec:E}

Let $G$ be a smooth commutative group scheme over $\Fq$.
The zeroth page of the Hochschild-Serre spectral sequence
is a double complex $E^{\bullet, \bullet}$ defined by
\[
E^{i,j} = C^i(\Weil{}, C^j(G(\bFq), \EEx));
\]
see \cite{vakil:Algebraic_Geometry}*{\S 1.7}, expanding on \cite{weibel:Homological_Algebra}*{Ch 5 and \S 7.5}.
The standard derivative on cochains yields two derivatives,
\begin{align*}
d_G &: E^{i,j} \to E^{i,j+1} \quad \mbox{and} \\
d_{\Weil{}} &: E^{i,j} \to E^{i+1,j};
\end{align*}
we use the first as the derivative $d_0$ on the zeroth page, and the second to induce $d_1$.
Combining them also yields a derivative $d = d_G + (-1)^j d_{\Weil{}}$ on the total complex
\[
E^n_G = \bigoplus_{i+j=n} E^{i,j}.
\]
The machinery of spectral sequences gives us a sequence of pages $E_r^{i,j}$, converging to a page $E_{\infty}^{i,j}$.
We summarize the key properties of this spectral sequence in the following proposition.

\begin{proposition} In the spectral sequence defined above,
\begin{enumerate}
\item the second page is given by $E_2^{i,j} = \Hh^i(\Weil{}, \Hh^j(G(\bFq), \EEx))$,
\item there is an isomorphism $\Hh^n(\Weil{} \ltimes G(\bFq), \EEx) \cong \Hh^n(E_G^\bullet)$, and
\item there is a filtration $\Hh^n(\Weil{} \ltimes G(\bFq), \EEx) = F_n \supset \cdots \supset F_{-1} = 0$, with $F_i / F_{i-1} \cong E_{\infty}^{i, n-i}$.
\end{enumerate}
\end{proposition}

Moreover, since $\Weil{} \cong \ZZ$ has cohomological dimension $1$, $E_2^{i,j} = 0$ for $i > 1$ and the
sequence degenerates at the second page: $E_{\infty}^{i,j} = E_2^{i,j}$. We obtain the following corollary:

\begin{corollary}\label{cor:spectral_ses}
There is a short exact sequence
 \[
    0 \to
    \Hh^0(\Weil{},\Hh^2(\bG,\EEx)) \to
    \Hh^2(E^\bullet_G) \to
    \Hh^1(\Weil{},\Hh^1(\bG,\EEx)) \to
    0.
 \]
\end{corollary}

This sequence will play a key role in understanding the kernel of $\TrFrob{G}$, as described in the next few sections.
For this application, we need a good understanding of these maps to and from the total complex.

\begin{proposition} \label{prop:ses_desc}
Consider the short exact sequence in Corollary~\ref{cor:spectral_ses}.
\begin{enumerate}
\item Every class $[\alpha\oplus\beta\oplus\gamma] \in \Hh^2(E^\bullet_G)$ is cohomologous to one with $\gamma=0$.
\item The map $\Hh^2(E^\bullet_G) \to \Hh^1(\Weil{},\Hh^1(\bG,\EEx))$ is given by $[\alpha\oplus\beta\oplus 0] \mapsto [\beta]$.
\item Suppose $a \in Z^2(\bG, \EEx)$ represents a class in $\Hh^2(\bG,\EEx)$ fixed by Frobenius.
The map $\Hh^0(\Weil{},\Hh^2(\bG,\EEx)) \to \Hh^2(E^\bullet_G)$ is given by $[a] \mapsto [a \oplus 0 \oplus 0]$.
\end{enumerate}
\end{proposition}
\begin{proof}
Since $\Hh^2(\Weil{}, C^0(\bG, \EEx)) = 0$, we may find a $\gamma_1 \in C^1(\Weil{}, C^0(\bG, \EEx))$ with $d_{\Weil{}}\gamma_1 = \gamma$.
Subtracting $d \gamma_1$ from $\alpha\oplus\beta\oplus\gamma$, we may assume that $\gamma = 0$.

The latter two claims follow from tracing through the definition of latter pages in the spectral sequence.
\end{proof}

\subsection{From character sheaves to the total complex}\label{ssec:S}

Let $G$ be a smooth commutative group scheme over $\Fq$.
In this section we define a group homomorphism
\[
S_G : \CSiso{G} \to \Hh^2(E^\bullet_G).
\] 
Let $\cs{L} = (\gcs{L},\mu,\phi)$ be a character sheaf on $G$.
For each geometric point $x\in \bG$, choose a basis $\{ v_x \}$ for $\gcs{L}_x$.
Through this choice, $\cs{L}$ determines functions
\begin{align*}
a : \bG\times \bG &\to \EEx & b : \bG &\to \EEx \\
\mu_{x,y}(v_{x+y}) &= a(x,y) v_x \otimes v_y & \phi_x(v_{\Frob{G}(x)}) &= b(x) v_x.
\end{align*}
Condition~\ref{CS.2} implies that
\begin{equation}\label{2-cocyle}
a(x+y,z) a(x,y) = a(x,y+z) a(y,z)
\end{equation}
for all $x,y,z\in \bG$, so $a \in Z^2(\bG,\EEx)$.  Similarly, condition~\ref{CS.3} gives
\begin{equation}\label{eq:nohom}
\frac{a(\Frob{G}(x),\Frob{G}(y))}{a(x,y)} =  \frac{b(x+y)}{b(x) b(y)}
\end{equation}
for $x, y \in \bG$.
Let $\alpha \in C^0(\Weil{},C^2(\bG,\EEx)$ be the $0$-cochain corresponding to $a$ and let $\beta\in C^1(\Weil{},C^1(\bG,\EEx)$
be the cocycle such that $\beta(\Frob{})$ is $b$.  We will write both $\alpha$ and $\beta$ additively.
Then
\[
d_G\alpha =0, \qquad\qquad
d_{\Weil{}} \alpha = d_{G} \beta,\qquad\qquad
d_{\Weil{}} \beta =0;
\]
in other words,
\[\alpha\oplus \beta \in Z^2(E^\bullet_G).\]
Although the cocycle $\alpha\oplus \beta$ is not well defined by $\cs{L}$, its class in $\Hh^2(E^\bullet_G)$ is.
To see this, let $\{ v'_x \in \gcs{L}_x^\times \tq x \in \bG\}$ be another choice and let $\alpha'\oplus \beta' \in Z^2(E^\bullet_G)$
be defined by $\cs{L}$ and this choice, as above.
Now let $\delta \in C^0(\Weil{},C^1(\bG,\EEx))$ correspond to the function $d : \bG\to \EEx$ defined by $v'_x = d(x) v_x$.
Chasing through \ref{CS.2} and \ref{CS.3}, we find
\[
\alpha'\oplus\beta' = \alpha\oplus\beta + d\delta,
\]
so the class $[\alpha\oplus\beta]$ of $\alpha\oplus\beta$ in $\Hh^2(E^\bullet_G)$ is independent of the choice made above.
It is also easy to see that $[\alpha\oplus\beta] = [\alpha_0\oplus\beta_0]$ when $\cs{L} \iso \cs{L}_0$,
which concludes the definition of the function
\begin{align*}
S_G : \CSiso{G} &\to \Hh^2(E^\bullet_G)\\
[\cs{L}] &\mapsto [\alpha\oplus \beta].
\end{align*}
It is also easy to see that $[\alpha_1\oplus\beta_1] + [\alpha_2\oplus\beta_2] = [\alpha_3\oplus\beta_3]$
when $\cs{L}_3 = \cs{L}_1\otimes \cs{L}_2$, so $S_G$ is a group homomorphism.

\begin{proposition}\label{prop:SGiso}
If $G$ is \'etale then $S_G:  \CSiso{G} \to \Hh^2(E^\bullet_G)$ is an isomorphism.
\end{proposition}
\begin{proof}
Suppose $[\cs{L}] \in \CSiso{G}$ with $S_G([\cs{L}]) = [\alpha \oplus \beta] = 0$,
so that $\alpha \oplus \beta = d\sigma$ for some $\sigma \in C^0(\Weil{},C^1(\bG,\EEx) = C^1(\bG,\EEx)$.
For each $x\in \bG$, define $\sigma_x : \gcs{L}_x \to \EE$ by $\sigma_x : v_x \mapsto \sigma(x)$.
Then the indexed set of isomorphisms $\{ \sigma_x : \gcs{L}_x \to \EE \tq x\in \bG\}$
defines an isomorphism $\cs{L} \to (\EE)_G$.
Since $\cs{L} = 0 \in \CSiso{G}$, $S_G$ is injective.

To see that $S_G$ is surjective, begin with $\alpha\oplus\beta\oplus 0 \in Z^2(E^\bullet_G)$.
Since $d_{\Weil{}} \beta = 0$, we may define $a = \alpha \in C^2(\bG,\EEx)$ and
$b = \beta(\Frob{}) \in C^1(\bG,\EEx)$, which are related to $\alpha$ and $\beta$ as above.
Set $\gcs{L}_x = \EE$, define $\mu_{x,y} : \gcs{L}_{x+y} \to \gcs{L}_x\otimes\gcs{L}_y$
by $\mu_{x,y}(1) = a(x,y) (1\otimes 1)$ and $\phi_x : \gcs{L}_{\Frob{G}(x)} \to \gcs{L}_x$ by $\phi_x(1)= b(x)$.
Then \ref{CS.1} holds since $d_G \alpha =0$ and \ref{CS.2} holds since $d_{\Weil{}}\alpha =d_G \beta$.
Tracing the construction backward, we have defined a character sheaf $\cs{L}$ on $G$ with
$S_G(\cs{L}) = [\alpha\oplus\beta\oplus 0]$, showing that $S_G$ is surjective.
\end{proof}

\subsection{Objects in the \'etale case}\label{ssec:SandT}

In this section we fit the group homomorphisms $\TrFrob{G}$ and $S_G$ into a commutative diagram,
determining the kernel and cokernel of $\TrFrob{G}$ when $G$ is an \'etale commutative group scheme over $\Fq$.
We begin with a simple, general result relating duals, invariants and coinvariants.

\begin{lemma} \label{lem:dual-inv}
Let $X$ be an abelian group equipped with an action of $\Weil{}$.  Then
\begin{align*}
 (X^*)_{\Weil{}} &\to (X^{\Weil{}})^* \\
 [f] &\mapsto f|_{X^{\Weil{}}}
\end{align*}
is an isomorphism.
\end{lemma}

\begin{proof}
We can describe $X^{\Weil{}}$ as the kernel of the map $X \xrightarrow{\Frob{}-1} X$;
let $Y = (\Frob{}-1)X$ be the augmentation ideal.  Dualizing the sequence
\[
 0 \to X^{\Weil{}} \to X \to Y \to 0
\]
yields
\[
 0 \to Y^* \to X^* \to (X^{\Weil{}})^* \to \Ext^1_\ZZ(Y, \EEx).
\]
Since $\Ext^1_\ZZ(-,\EEx)$ vanishes, we get a natural isomorphism from the cokernel of $Y^* \xrightarrow{\Frob{}-1} X^*$ to $(X^{\Weil{}})^*$.
\end{proof}

\begin{proposition}\label{prop:sur_etale}
If $G$ is \'etale, then $\TrFrob{G} : \CSiso{G} \to G(\Fq)^*$ is surjective
and split.
\end{proposition}
\begin{proof}
Pick $\chi \in G(\Fq)^*$. 
Let $[\beta]\in \Hh^1(\Weil{},\bG^*)$ be the class corresponding to $\chi$ under Lemma~\ref{lem:dual-inv}.
Every representative cocycle $\beta \in Z^1(\Weil{},\bG^*)$ determines a homomorphism
$\beta(\Frob{}) : G(\bFq)\to \EEx$ such that $\beta(\Frob{})\vert_{G(\Fq)} = \chi$.
Set $\gcs{L}_x = \EE$ for every $x\in G(\bFq)$.
Define $\mu_{x,y} : \gcs{L}_{x+y} \to \gcs{L}_x\otimes \gcs{L}_y$ by $\mu_{x,y}(1) = 1 \otimes 1$ and
$\phi_{x} : \gcs{L}_{\Frob{}(x)} \to \gcs{L}_x$ by $\phi_{x}(1) = \beta(\Frob{})(x)$.
Since $\beta(\Frob{}) : G(\bFq) \to \EEx$ is a group homomorphism,
condition \eqref{eq:nohom} is satisfied with $a =1$.
So $\cs{L} = (\gcs{L}, \mu, \phi)$
is a character sheaf with $\trFrob{\cs{L}} = \chi$.
This shows that $\TrFrob{G}$ is surjective.

Now let $\beta' \in Z^1(\Weil{},\bG^*)$ be another representative for $[\beta]$
so $\beta-\beta' = d_{\Weil{}} \delta$ for some $\delta \in C^0(\Weil{},\bG^*)$ defining $d \in \Hom(G(\bFq),\EEx)$.
Let $\cs{L}'$ be the character sheaf on $G$ defined by $\beta'$, as above.
For each $x\in G(\bFq)$, define $\brho_x :\cs{L}_x\to \cs{L}'_x$ by $\brho_x(1) = d(x)$.
The collection of isomorphisms $\{ \brho_x \tq x\in G(\bFq)\}$ satisfies condition~\ref{CS.4},
so it defines a morphism $\rho : \cs{L}\to \cs{L}'$, which is clearly an isomorphism. 
We have now defined a section of $\TrFrob{G}$. 

Now suppose $\chi_1, \chi_2 \in G(\Fq)^*$. Pick cocycles $\beta_1,\beta_2\in Z^1(\Weil{},\bG^*)$ and
construct character sheaves $\cs{L}_1$ and $\cs{L}_2$ on $G$ as above. Since $\cs{L}_1\otimes \cs{L}_2$
is exactly the character sheaf built from the cocycle $\beta_1\cdot \beta_2$, and since
$\trFrob{\cs{L}_1\otimes \cs{L}_2} = \trFrob{\cs{L}_1}\cdot \trFrob{\cs{L}_2}$, the section of $\TrFrob{G}$ defined here is a homomorphism.
\end{proof}

\begin{proposition} \label{prop:etale-iso}
 If $G$ is \'etale then the map $S_G : \CSiso{G}\to \Hh^2(E^\bullet_G)$ induces an isomorphism of split short exact sequences
\[
\begin{tikzcd}
 0 \arrow{r} & \ker \TrFrob{G} \arrow{d} \arrow{r} & \CSiso{G}\arrow{d}{S_G} \arrow{r}{\TrFrob{G}} \arrow{r} & G(\Fq)^* \arrow{d} \arrow{r} & 0\\
  0 \arrow{r} & \Hh^0(\Weil{},\Hh^2(\bG,\EEx)) \arrow{r} & \Hh^2(E^\bullet_G) \arrow{r} & \Hh^1(\Weil{},\Hh^1(\bG,\EEx)) \arrow{r} & 0.
 \end{tikzcd}
 \]
\end{proposition}
\begin{proof}
This result follows from Propositions~\ref{prop:ses_desc}, \ref{prop:SGiso} and \ref{prop:sur_etale}.
\end{proof}

\begin{definition}
We call a character sheaf $\cs{L}$ on $G$ \emph{invisible} if it is nontrivial and $\TrFrob{G}(\cs{L}) = 1$.
\end{definition}

The proposition gives a method for determining whether a given $G$ admits invisible character sheaves.

\begin{remark} \label{rem:H2}
Recall the K\"unneth formula in group cohomology \cite{brown:CohomologyGrps}*{Prop. I.0.8}:
if $A$ and $A'$ are groups and $M$ and $M'$ are abelian groups with $M$ $\ZZ$-free, then
\[
\Hh^n(A \times A', M \otimes M') \cong \bigoplus_{i+j=n} \Hh^i(A, M) \otimes \Hh^j(A', M') \ \ \oplus
\bigoplus_{i+j=n+1} \Tor_1^{\ZZ}\left(\Hh^i(A, M), \Hh^j(A', M')\right).
\]
Now suppose $\bG = \ZZ^r \times \prod_{i = 1}^m \ZZ/N_i\ZZ$ is an arbitrary finitely generated abelian group,
with $N_i \mid N_{i+1}$.  Then the K\"unneth formula implies that
\begin{equation} \label{eq:H2comp}
\Hh^2(\bG, \EEx) \cong \left(\EEx\right)^{r(r-1)/2} \times \prod_{i=1}^m (\ZZ/N_i\ZZ)^{m+r-i}.
\end{equation}
We see that $\Hh^2(\bG, \EEx)$ is trivial if and only if $\bG$ is cyclic.  Of course, $\Hh^0(\Weil{}, \Hh^2(\bG, \EEx))$
may or may not be trivial, even when $\Hh^2(\bG, \EEx)$ is non-trivial.
\end{remark}

\begin{example}\label{eg:H2}
Consider the simplest non-trivial case, where $\bG = \{1, i, j, k\} \cong \ZZ/2\ZZ \times \ZZ/2\ZZ.$  Using \eqref{eq:H2comp},
we have $\Hh^2(\bG, \EEx) \cong \ZZ/2\ZZ$, on which $\Weil{}$ must act trivially, regardless of its action on $\bG$ itself.
The non-trivial element corresponds to the extension
\begin{equation} \label{eq:Qdef}
1 \to \EEx \to Q \to \bG \to 1,
\end{equation}
where $Q = \{c + c_ii + c_jj + c_kk \tq c,c_i,c_j,c_k \in \EE \mbox{ with exactly one nonzero}\}$ is a subgroup of the quaternion algebra over $\EE$.
Let $a$ be a $2$-cocycle corresponding to this extension, with values in $\{\pm 1\}$. When $\Frob{G}$ acts trivially on $\bG$,
any homomorphism $b : \bG \to \EEx$ will satisfy \eqref{eq:nohom}, and the corresponding
$\alpha \oplus \beta$ are non-cohomologous in $\Hh^2(E_G^\bullet)$.
When $\Frob{G}$ exchanges $i$ and $j$, then we may take $b(1) = 1, b(i) = -1$ and $b(j) = b(k) = \pm 1$, up to coboundaries.
Finally, when $\Frob{G}$ cycles $i$, $j$ and $k$, any homomorphism $b : \bG \to \EEx$ will satisfy \eqref{eq:nohom},
but now the corresponding $\alpha \oplus \beta$ are all cohomologous in $\Hh^2(E_G^\bullet)$.  In each case, we may produce
an explicit character sheaf from the listed $a$ and $b$.

Note that these character sheaves arise from discrete isogenies, as in Section \ref{ssec:discrete_isogenies}.
Let $\bH$ be the quaternion group of order $8$: the subgroup of $Q$ with $c, c_i, c_j, c_k \in \{\pm 1\}$.
The sequence \eqref{eq:Qdef} is the pushforward of
\[
1 \to \{\pm 1\} \to \bH \to \bG \to 1
\]
along the inclusion $\{\pm 1\} \hookrightarrow \EEx$.  Note that these character sheaves arise from a non-commutative cover of $\bG$,
justifying the inclusion of such covers in the definition of a discrete isogeny.
\end{example}

\subsection{On the necessity of working with Weil sheaves}\label{ssec:bS}

In this section we justify the appearance of Weil sheaves in Definition~\ref{def:CS}.

\begin{proposition}\label{prop:bounded-etale}
Let $G$ be a commutative \'etale group scheme over $\Fq$.
Then the image of $\bCS(G)$ under $\TrFrob{G} : \CS(G) \to G(\Fq)^*$ is $\Hom(G(\Fq),\ZEx)$.
\end{proposition}

\begin{proof}
Objects in $\bCS(G)$ may be described by a small modification to the technique used in Sections~\ref{ssec:E} and \ref{ssec:S}. 
Set $F^{i,j} \ceq C^i_{\text{cts}} (\Gal(\bFq/\Fq), C^j(G(\bFq), \EEx))$.
Then the results of Section~\ref{ssec:E} adapt to give a short exact sequence in continuous Galois cohomology
 \[
    0 \to
    \Hh^0(\Fq,\Hh^2(\bG,\EEx)) \to
    \Hh^2(F^\bullet_G) \to
    \Hh^1(\Fq,\Hh^1(\bG,\EEx)) \to
    0,
 \]
for which the maps are given by the analogues of Proposition~\ref{prop:ses_desc}.
Moreover, using \cite{deligne-katz:SGA7.2}*{Expos\'e XIII, Rappel 1.1.3} we see that Proposition~\ref{prop:SGiso}
adapts to provide an isomorphism $\bCSiso{G}\to \Hh^2(F^\bullet_G)$ compatible with $\bCS(G) \to \CS(G)$ and with the
the canonical map of exact sequences 
 \[
\begin{tikzcd}
    0 \arrow{r} &
    \Hh^0(\Fq,\Hh^2(\bG,\EEx)) \arrow{r} \arrow{d} &
    \Hh^2(F^\bullet_G) \arrow{r} \arrow{d} &
    \Hh^1(\Fq,\Hh^1(\bG,\EEx)) \arrow{r} \arrow{d} &
    0
\\
    0 \arrow{r} & 
    \Hh^0(\Weil{},\Hh^2(\bG,\EEx)) \arrow{r} & 
    \Hh^2(E^\bullet_G) \arrow{r} & 
    \Hh^1(\Weil{},\Hh^1(\bG,\EEx)) \arrow{r} &
    0.
\end{tikzcd}
 \]
In this way, Proposition~\ref{prop:bounded-etale} is now reduced to the claim
\[
\Hh^1(\Fq,\Hh^1(\bG,\EEx)) = \Hom(G(\Fq),\ZEx).
\]
To see that, one may argue as follows. 
Pick $i\in \pi_0(G)$ and let $G^i \hookrightarrow G$ be the corresponding connected component. 
Pick a geometric point $x$ on $G^i$ and observe that since $G^i$ is connected as a $\Fq$-scheme,
$G^i(\bFq)$ is canonically identified with the $\Gal(\bFq/\Fq)$-orbit of $x$. 
We remark that while $G^i$ is defined over $\Fq$, the set $G^i(\Fq)$ is non-empty only when $G^i(\bFq) = \{ x\}$.
Since $\Hh^1(\bG,\EEx) = \Hom(\bG,\EEx)$, evaluation $\chi \mapsto \chi(x)$ defines $\Hh^1(\Fq,\Hh^1(\bG,\EEx)) \to \Hh^1(\Fq,\EEx)$. 
By continuity, $\Hh^1(\Fq,\EEx) = \Hh^1(\Fq,\ZEx)$.
Letting $i$ range over $\pi_0(G)$ we conclude that $\Hh^1(\Fq,\Hh^1(\bG,\EEx)) = \Hh^1(\Fq,\Hh^1(\bG,\ZEx))$.
When adapted to abelian groups with continuous action of $\Gal(\bFq/\Fq)$, the strategy of the proof of
Lemma~\ref{lem:dual-inv} gives $\Hh^1(\Fq,\Hh^1(\bG,\ZEx)) =  \Hom(G(\Fq),\ZEx)$, concluding the proof.
\end{proof}

Proposition~\ref{prop:bounded-etale} reveals the necessity of working with Weil sheaves in Definition~\ref{def:CS}:
one cannot geometrize all characters of $G(\Fq)$ using local systems on $G$, for general smooth commutative groups schemes $G$.
Proposition~\ref{prop:bounded-etale} is extended to all smooth commutative groups schemes in Section~\ref{ssec:revisited}.

\begin{example}\label{eg:}
Consider the case when $G$ is the \'etale group scheme $\ZZ$ over $\Fq$ with $\Frob{G}$ trivial.
If $\chi : \ZZ \to \EEx$ is the character of $G(\Fq)$ determined by $\chi(1) = \ell$
and if $\cs{L}$ is a character sheaf on $G$ in the isomorphism class
corresponding to $\chi$ under Proposition~\ref{prop:sur_etale},
then $\cs{L}$ does not descend to $G$, since the image of $\chi$ is not
bounded.
If $\chi' : \ZZ \to \EEx$ is the character of $G(\Fq)$ determined by $\chi'(1) = 1+\ell$ and if $\cs{L}'$
corresponds to $\chi'$ under Proposition~\ref{prop:sur_etale},
then $\cs{L}'$ does descend to $G$, since the image of $\chi'$ is
bounded.
However, $\cs{L}'$ is not defined by a discrete isogeny (\S \ref{ssec:discrete_isogenies}). 
If $\chi'' : \ZZ \to \EEx$ is the character of $G(\Fq)$ determined by $\chi''(1) = \zeta$, a root of unity in
$\EEx$, and if $\cs{L}''$ corresponds to $\chi''$ under Proposition~\ref{prop:sur_etale},
then $\cs{L}''$ is defined by a discrete isogeny. 
\end{example}

\subsection{Morphisms in the \'etale case}\label{ssec:mor-etale}
 
A complete understanding of the morphisms in $\CS(G)$ also requires a description of the
automorphisms of an arbitrary character sheaf $\cs{L}$.

\begin{proposition}\label{prop:autornaught_etale}
Let $G$ be an \'etale commutative group scheme over $\Fq$.
If $\cs{L}$ and $\cs{L}'$ are character sheaves on $G$ then
every $\rho\in \Hom(\cs{L},\cs{L}')$ is either trivial or an isomorphism. Moreover, the trace map induces an isomorphism of groups
\[
\Aut(\cs{L}) \to \Hom(G(\bFq)_{\Weil{}}, \EEx).
\]
\end{proposition}

\begin{proof}
We have already seen, in Lemma~\ref{lem:autornaught}, that every $\rho\in \Hom(\cs{L},\cs{L}')$ is either trivial or an isomorphism.
Now suppose $\rho \in \Aut(\cs{L})$.
The second diagram in \ref{cs.4} shows that the association $x \mapsto \brho_x$ is a homomorphism
from $G(\bFq)$ to $\EEx$ and the first diagram in \ref{cs.4} shows that it factors through $G(\bFq) \to G(\bFq)_{\Weil{}}$.  

Conversely, if $\rho : G(\bFq)_{\Weil{}} \to \EEx$ is any homomorphism, then defining $\brho_x$ as multiplication
by $\rho(x)$ will define a morphism $\gcs{L} \to \gcs{L}'$ satisfy the two diagrams in \ref{cs.4}.  

Composition of morphisms corresponds to pointwise multiplication in this correspondence, showing that the
resulting bijection is actually a group isomorphism.
\end{proof}

\section{Character sheaves on smooth commutative group schemes over finite fields}\label{sec:main}

\subsection{Restriction to the identity component} \label{ssec:restriction}

Consider the short exact sequence
defining the component group scheme for $G$:
\begin{equation}\label{eq:pi0}
\begin{tikzcd}
0 \rar & G^0 \arrow{r}{\iota_0} & G \arrow{r}{\pi_0} & \pi_0(G) \rar & 0.
\end{tikzcd}
\end{equation}
Since $\pi_0(G)$ is an \'etale commutative group scheme -- and thus smooth --
Lemma~\ref{lem:pullback} implies that \eqref{eq:pi0} defines a sequence of functors
\begin{equation}\label{eq:pi1}
\begin{tikzcd}
\CS(0) \rar & \CS(\pi_0(G)) \arrow{r}{\pi_0^*} & \CS(G) \arrow{r}{\iota_0^*} & \CS(G^0) \rar & \CS(0)
\end{tikzcd}
\end{equation}
and therefore, after passing to isomorphism classes, a sequence of abelian groups
\begin{equation}\label{eq:pi2}
\begin{tikzcd}
0 \rar &
\CSiso{\pi_0(G)} \arrow{r}{\pi_0^*} & \CSiso{G} \arrow{r}{\iota_0^*} & \CSiso{G^0} \rar & 0.
\end{tikzcd}
\end{equation}
 Note that we found the groups $\CSiso{\pi_0(G)}$ and $\CSiso{G^0}$
in Sections~\ref{ssec:SandT} and \ref{ssec:connected}, respectively.
We will shortly see that \eqref{eq:pi2} is exact.

\begin{lemma}\label{lemma:ext}
Every discrete isogeny to $G^0$ extends to a discrete
isogeny to $G$ inducing an isomorphism on component groups.
\end{lemma}

\begin{proof}
Let $\pi: B \to G^0$ be a discrete isogeny, and set $A \ceq \ker \pi$.
  We will find a discrete isogeny $f: H\to G$
  such that that $H^0 = B$, $f^0 =\pi$ and
  $\pi_0(f) : \pi_0(H)\to \pi_0(G)$ is an isomorphism of component
  groups.  Namely, we will fit $\pi$ into the following diagram,
  \begin{equation}\label{extension-diagram}
  \begin{tikzcd}
  A \arrow{r} \dar & A \dar \\
  B \rar \dar[swap]{\pi} & H \rar \dar[swap]{f} & \pi_0(H) \arrow{d}[below,rotate=90]{\sim}[swap]{\pi_0(f)} \\
  G^0 \rar & G \rar & \pi_0(G),
  \end{tikzcd}
  \end{equation}
  where all rows and columns are exact and all maps are defined over
  $\Fq$.  We will do so by passing back and forth between group
  schemes over $\Fq$ and their $\bFq$-points.

  Extensions of $G^0(\bFq)$ by $A(\bFq)$ with $\Weil{}$-equivariant maps, such as $B(\bFq)$,
  correspond to classes in $\Ext^1_{\ZZ[\Weil{}]}(G^0(\bFq), A(\bFq))$.
  Similarly, extensions of $G(\bFq)$ by $A(\bFq)$ with $\Weil{}$-equivariant maps correspond to
  classes in $\Ext^1_{\ZZ[\Weil{}]}(G(\bFq), A(\bFq))$.  The map
  $G^0(\bFq) \to G(\bFq)$ induces a homomorphism
  \[
  \Ext^1_{\ZZ[\Weil{}]}(G(\bFq), A(\bFq)) \to \Ext^1_{\ZZ[\Weil{}]}(G^0(\bFq), A(\bFq))
  \]
  fitting into the long exact sequence 
  \[
  \Ext^1_{\ZZ[\Weil{}]}(G(\bFq), A(\bFq)) \to \Ext^1_{\ZZ[\Weil{}]}(G^0(\bFq), A(\bFq)) \to \Ext^2_{\ZZ[\Weil{}]}(\pi_0(G)(\bFq), A(\bFq))
  \]
  derived from applying
  the functor $\Hom(\mbox{---}, A(\bFq))$ to $G^0(\bFq) \to G(\bFq) \to \pi_0(G)(\bFq)$.
  Since $\Weil{} \cong \ZZ$ has cohomological dimension $1$ \cite{brown:CohomologyGrps}*{Ex. 4.3},
  $\Ext^2_{\ZZ[\Weil{}]}(\pi_0(G)(\bFq), A(\bFq))$ vanishes \cite{cartan-eilenberg:HomologicalAlgebra}*{Thm. 2.6}.

  We therefore have the existence of diagram \eqref{extension-diagram}
  at the level of $\bFq$-points.  This expresses $H(\bFq)$ as a
  disjoint union of translates of $B(\bFq)$; by transport of structure
  we may take $H$ to be a group scheme over $\bFq$.  Similarly, the
  restriction of $f$ to each component of $H$ is a morphism of
  schemes, and thus $f$ is as well.  Finally, the whole diagram
  descends to a diagram of $\Fq$-schemes since the $\bFq$-points of
  the objects come equipped with continuous $\Gal(\bFq/\Fq)$-actions and the
  morphisms are $\Gal(\bFq/\Fq)$-equivariant.
\end{proof}

We now wish to apply the results of Section~\ref{ssec:connected} to the identity component of $G$,
for which we must confirm that the identity component of $G$ is actually an algebraic group over $\Fq$.

\begin{lemma} \label{lem:G0alg-grp}
If $G$ is a commutative smooth group scheme over $\Fq$ then its identity component, $G^0$, is a connected algebraic group over $\Fq$.
\end{lemma}
\begin{proof}
 Since $G$ is a smooth group scheme over $\Fq$, its
 identity component $G^0$ is a connected smooth,
 group scheme of finite type over $\Fq$, reduced over some finite extension of $\Fq$
 \cite{vdGeer-Moonen:AbelianVarieties}*{3.17}.
 Since $\Fq$ is a finite field and hence perfect, $G^0$ is actually reduced over $\Fq$
 \cite{EGAIV2}*{Prop 6.4.1}.  Since every group scheme over a field is separated
 \cite{vdGeer-Moonen:AbelianVarieties}*{3.12},
 it follows that $G^0$ is a connected algebraic group.
\end{proof}

\begin{proposition}\label{prop:restriction}
The restriction functor $\iota_0^* : \CS(G)\to \CS(G^0)$ is essentially surjective.
\end{proposition}

\begin{proof}
  By Lemma~\ref{lem:G0alg-grp} and Proposition~\ref{prop:connected}, every
  character sheaf on $G^0$ is isomorphic to $(\pi_! \EE)_\psi$ for some discrete isogeny $\pi : B \to G^0$ and character $\psi : \ker \pi \to \EEx$.
  So to prove the proposition it suffices to show that $(\pi_! \EE)_\psi$ extends to a character sheaf on $G$.
 By Lemma~\ref{lemma:ext}, there is an extension of the
 discrete isogeny $\pi : B \to G^0$ to a discrete isogeny $f : H \to G$
 such that $\pi_0(f) : \pi_0(H)\to \pi_0(G)$ is an isomorphism.
 Then $(f_! \EE)_\psi$ is a character sheaf on $G$ and
 $(f_! \EE)_\psi\vert_{G^0} \iso (\pi_! \EE)_\psi$.
\end{proof}

\subsection{The component group sequence} \label{ssec:component}

\begin{lemma}\label{lem:extension}
The group homomorphism $\pi_0^*: \CSiso{\pi_0(G)} \to \CSiso{G}$ is injective.
\end{lemma}
\begin{proof}
Let $\cs{L}$ be a character sheaf on $\pi_0(G)$ and let $\rho : \pi_0^*\cs{L} \to (\EE)_{G}$ be an isomorphism in $\CS(G)$. 
For each $x\in \pi_0(\bG)$, set $\bG^x \ceq \pi_0^{-1}(x)$.
The restriction $\pi_0^*\gcs{L}\vert_{\bG^x}$ is the constant sheaf $(\gcs{L}_x)_{\bG^x}$ so the isomorphism
$\brho\vert_{\bG^x} : (\gcs{L}_x)_{\bG^x} \to (\EE)_{\bG^x}$ determines an isomorphism $\brho_x : \gcs{L}_x \to (\EE)_x$.
The collection $\{ \brho_x \tq x\in \pi_0(\bG) \}$ determines an isomorphism $\cs{L} \to (\EE)_{\pi_0(G)}$  in $\CS(\pi_0(G))$.
\end{proof}

\begin{proposition}\label{prop:middleexact}
 The sequence
 \[
  \begin{tikzcd}
  0 \rar & \CSiso{\pi_0(G)} \arrow{r}{\pi_0^*} & \CSiso{G} \arrow{r}{\iota_0^*} & \CSiso{G^0} \rar & 0.
  \end{tikzcd}
 \]
 is exact.
\end{proposition}

\begin{proof}
Exactness at $\CSiso{G^0}$ follows from Proposition~\ref{prop:restriction},
and exactness at $\CSiso{\pi_0(G)}$ from Lemma~\ref{lem:extension}.
Here we show that it is also exact at $\CSiso{G}$.
First note that $\iota_0^* \circ \pi_0^*$ is trivial by Lemma~\ref{lem:pullback}.
So it suffices to show that if $\cs{L} = (\gcs{L},\mu,\phi)$ is a character sheaf on $G$
with $\cs{L}\vert_{G^0} = (\EE)_{G^0}$ then $\cs{L}$ is in the essential image of $\pi_0^*$.

As above, set $\bG^x \ceq \pi_0^{-1}(x)$ for $x\in \pi_0(\bG)$.
Let $g, g'$ be geometric points in the same
geometric connected component $\bG^x$.
Set $a = g^{-1}g'$ and note that $a$ is a geometric point in $\bG^0$.
Let $\mu_{g,a} : \gcs{L}_{ga} \to \gcs{L}_g \otimes \gcs{L}_a$
be the isomorphism of vector spaces obtained by restriction of
$\mu : m^*\gcs{L} \to \gcs{L} \boxtimes \gcs{L}$ to the
geometric point $(g,a)$ on $\bG^x \times \bG^0$.
Since $\cs{L}\vert_{G^0} = (\EE)_{G^0}$,
the stalk of $\gcs{L}$ at $a$ is $\EE$.
In this way the pair of geometric points $g, g' \in \bG^x$
determines an isomorphism $\varphi_{g,g'} \ceq \mu_{g,a}^{-1}$
from $\gcs{L}_{g}$ to $\gcs{L}_{g'}$.
The isomorphisms $\varphi_{g,g'}: \gcs{L}_{g} \to \gcs{L}_{g'}$ are canonical
in the following sense: if $g,g'\in \bG^x$ and $h,h'\in \bG^y$
then it follows from \ref{CS.2} and \ref{CS.3}
that
 \begin{equation}\label{eq:CS}
  \begin{tikzcd}[column sep=40]
   \gcs{L}_{gh} \arrow{r}{\varphi_{gh,g'h'}} \arrow[swap]{d}{\mu_{g,h}}
  & \gcs{L}_{g'h'} \arrow{d}{\mu_{g',h'}}
  &\arrow[draw=none]{d}[pos=.4,description]{\text{\normalsize{and}}}
  &  \gcs{L}_{\Frob{}(g)} \arrow{r}{\varphi_{\Frob{}(g),\Frob{}(g')}} \arrow[swap]{d}{\phi_{g}} & \gcs{L}_{\Frob{}(g')} \arrow{d}{\phi_{g'}} \\
  \gcs{L}_{g} \otimes \gcs{L}_{h} \arrow{r}{\varphi_{g,g'}\otimes \varphi_{h,h'}}
  & \gcs{L}_{g'} \otimes \gcs{L}_{h'}
  & {}
  & \gcs{L}_{g} \arrow{r}{\varphi_{g,g'}} & \gcs{L}_{g'}
  \end{tikzcd}
 \end{equation}
both commute.

For each $x\in \pi_0(\bG)$, pick $g(x)\in \bG^x$
and set $\gcs{E}_x \ceq \gcs{L}_{g(x)}$.
Let $\phi_x : \gcs{E}_{\Frob{}(x)} \to \gcs{E}_x$
be the isomorphism of $\EE$-vector spaces obtained by composing
$\varphi_{g(\Frob{}(x)),\Frob{}(g(x))} : \gcs{L}_{g(\Frob{}(x))} \to \gcs{L}_{\Frob{}(g(x))}$
with $\phi_{g(x)} : \gcs{L}_{\Frob{}(g(x))} \to \gcs{L}_{g(x)}$.
For each pair $x,y\in \pi_0(\bG)$
let $\mu_{x,y} : \gcs{E}_{x+y}\to \gcs{E}_x\otimes \gcs{E}_y$
be the isomorphism of $\EE$-vector spaces obtained by composing
$\varphi_{g(x+y),g(x)g(y)} : \gcs{L}_{g(x+y)} \to \gcs{L}_{g(x)g(y)}$
with $\mu_{g(x),g(y)} : \gcs{L}_{g(x)g(y)} \to \gcs{L}_{g(x)}\otimes \gcs{L}_{g(y)}$.
Using \eqref{eq:CS}, it follows that \ref{CS.1}, \ref{CS.2} and \ref{CS.3} are satisfied for
$\cs{E} \ceq (\gcs{E}_x, \mu_{x,y}, \phi_x)$, thus defining a character sheaf on $\pi_0(G)$.

The pullback $\pi_0^*(\cs{E})$ of $\cs{E}$ along $\pi_0 : G \to \pi_0(G)$ is constant
on geometric connected components, with stalks given by
$(\pi_0^* \cs{E})_g = \cs{E}_{x}$ for all $g\in \bG^x$.  Thus both $\pi_0^*\cs{E}$ and $\cs{L}$
are constant on geometric connected components of $G$.
The choices above define isomorphisms
$\gcs{L}\vert_{\bG^x} \to  (\gcs{E}_{x})_{\bG^x}$ for each $x\in \pi_0(\bG)$.
The resulting isomorphism $\gcs{L} \to \pi_0^* \gcs{E}$ satisfies \ref{CS.4},
thus defining an isomorphism $\cs{L} \to \pi_0^* \cs{E}$ in $\CS(G)$.
\end{proof}

\subsection{The dictionary}
\label{ssec:snake}

We saw in Proposition~\ref{prop:functorialG} that $\TrFrob{G} : \CSiso{G} \to G(\Fq)^*$ is a functorial group homomorphism.
In this section we find the image and kernel of $\TrFrob{G}$.

\begin{theorem}\label{thm:snake}
  If $G$ is a smooth commutative group scheme over $\Fq$ then
  $\TrFrob{G} : \CSiso{G} \to G(\Fq)^*$ is surjective and has kernel canonically isomorphic to $\Hh^2(\pi_0(\bG),\EEx)^{\Weil{}}$:
    \[
  \begin{tikzcd}
0 \arrow{r} & \Hh^2(\pi_0(\bG),\EEx)^{\Weil{}} \arrow{r} & \CSiso{G} \arrow{r}{\TrFrob{G}} & G(\Fq)^* \arrow{r} & 0. 
  \end{tikzcd}
\]
is an exact sequence.
\end{theorem}

\begin{proof}
  Let
  \begin{equation}\label{eq:pre-snake}
  \begin{tikzcd}[row sep=20, column sep=20]
    0 \rar & \CSiso{\pi_0(G)} \rar \dar{\TrFrob{\pi_0(G)}}
    & \CSiso{G} \rar \dar{\TrFrob{G}} & \CSiso{G^0} \rar \dar{\TrFrob{G^0}} & 0\\
    0 \rar & \pi_0(G)(\Fq)^* \rar
    &  G(\Fq)^* \rar & G^0(\Fq)^* \rar & 0
  \end{tikzcd}
  \end{equation}
  be the commutative diagram of abelian groups obtained by applying
  Lemma~\ref{lem:pullback} to \eqref{eq:pi0}.
  The sequence of abelian groups
\[
  \begin{tikzcd}
    1 \rar & G^0(\Fq) \rar & G(\Fq) \rar & \pi_0(G)(\Fq) \rar & 0,
  \end{tikzcd}
\]
  is exact since $\Hh^1(\Fq,G^0) =0$ by Lemma~\ref{lem:G0alg-grp} and Lang's theorem on connected algebraic groups over finite fields \cite{lang:56a}.
  Since $\EEx$ is divisible, $\Hom(\ - \ ,\EEx)$ is exact and thus the dual sequence of
  character groups in \eqref{eq:pre-snake} is exact.
  The upper row in \eqref{eq:pre-snake} is exact by Proposition~\ref{prop:middleexact}.
  Now Lemma~\ref{lem:G0alg-grp} and Proposition~\ref{prop:connected}
  imply that $\ker \TrFrob{G^0} =0$ and $\coker \TrFrob{G^0}=0$,
  while Proposition~\ref{prop:etale-iso} gives $\ker \TrFrob{\pi_0(G)} \iso \Hh^0(\Weil{},\Hh^2(\pi_0(\bG),\EEx))$
  and $\coker \TrFrob{\pi_0(G)}=0$.
It now follows from the snake lemma
 \begin{equation}\label{eq:snake}
  \begin{tikzcd}[row sep=20, column sep=15]
    0 \arrow[dotted]{r} & \ker \TrFrob{\pi_0(G)} \arrow[dotted]{r} \dar & \arrow{d} \ker \TrFrob{G} \arrow[dotted]{r}
    & \ker \TrFrob{G^0} =0 \arrow[dotted, out=-10, in=170]{dddll} \dar & \\
    0 \rar & \CSiso{\pi_0(G)} \rar \dar{\TrFrob{\pi_0(G)}}
    & \CSiso{G} \rar \dar{\TrFrob{G}} & \CSiso{G^0} \rar \dar{\TrFrob{G^0}} & 0\\
    0 \rar & \pi_0(G)(\Fq)^* \rar \dar
    & \arrow{d} G(\Fq)^* \rar & G^0(\Fq)^* \rar \dar & 0\\
   &  \coker \TrFrob{\pi_0(G)} =0 \arrow[dotted]{r} & \coker \TrFrob{G} \arrow[dotted]{r} & \coker \TrFrob{G^0} =0 & 
  \end{tikzcd}
  \end{equation}
that $\coker \TrFrob{G} =0$
and $\ker \TrFrob{\pi_0(G)} \to \ker \TrFrob{G}$ is an isomorphism.
This gives the promised short exact sequence
\[
  \begin{tikzcd}
0 \arrow{r} & \Hh^2(\pi_0(\bG),\EEx)^{\Weil{}} \arrow{r} & \CSiso{G} \arrow{r}{\TrFrob{G}} & G(\Fq)^* \arrow{r} & 0. 
  \end{tikzcd}
\]
\end{proof}

\begin{remark}\label{rem:rats}
Although $\TrFrob{\pi_0(G)}$ is split and $\TrFrob{G^0}$ is an isomorphism, we do not know if $\TrFrob{G}$ is split,
in general. Surjectively of $\TrFrob{G}$ shows that every $\ell$-adic character of $G(\Fq)$ admits a geometrization,
but without a splitting for $\TrFrob{G}$ we do not know how to make this geometrization canonical.
\end{remark}

\subsection{Descent, revisited}\label{ssec:revisited}

We now extend Proposition~\ref{prop:bounded-etale} to all smooth commutative group schemes over $\Fq$.

\begin{proposition} \label{prop:bounded}
Let $G$ be a smooth commutative group scheme over $\Fq$. 
Then $\cs{L}\in \CS(G)$ descends to $G$ if and only if $\trFrob{\cs{L}} : G(\Fq) \to \EEx$ has bounded image.
\end{proposition}
\begin{proof} 
By Lemma~\ref{lem:G0alg-grp}, the identity component $G^0$ is a connected algebraic group over $\Fq$. 
It follows from Proposition~\ref{prop:connected} that the restriction of $\cs{L}$ to $G^0$ descends to $G$. 
Also, since $G^0(\Fq)$ is finite, the image of $\trFrob{\cs{L}} : G(\Fq) \to \EEx$ is a finite subgroup and therefore has bounded image.  
If $\chi \in G(\Fq)^*$ then there is some finite-image character $\chi_0$
with the same restriction to $G^0(\Fq)$ since $G^0(\Fq)$ is lies inside the torsion part of
the finitely generated abelian group $G(\Fq)$.  Therefore $\chi$ is bounded 
if and only if $\chi \cdot \chi_0^{-1}$ is bounded.  But $\chi \cdot \chi_0^{-1}$ descends
to a character of $\pi_0(G)$.
Thus, it is enough to prove Corollary~\ref{prop:bounded}
for \'etale group schemes $G$, which is done in Proposition~\ref{prop:bounded-etale}.
\end{proof}

Proposition~\ref{prop:bounded} shows that the full subcategory
$\bCS(G) \subset \CS(G)$ is not an equivalence, for general smooth commutative group schemes $G$.
Again we see the necessity of working with Weil sheaves in Definition~\ref{def:CS}.

\subsection{Morphisms of character sheaves} \label{ssec:CSmor}

\begin{theorem}\label{thm:autornaught}
Let $G$ be a smooth commutative group scheme over $\Fq$.
There is a canonical isomorphism
\[
\Aut(\cs{L}) \iso \Hom(\pi_0(\bG)_{\Weil{}},\EEx).
\]
\end{theorem}
\begin{proof} 
Fix $\cs{L} = (\gcs{L},\mu,\phi)$ and consider the group homomorphism from $\Aut(\cs{L})$ to $\Hom(\bG_{\Weil{}},\EEx)$
defined in the proof of Proposition~\ref{prop:autornaught_etale}.
This homomorphism is injective because morphisms of sheaves are determined by the linear transformations induced on stalks.  
Homomorphisms in the image of $\Aut(\cs{L}) \to \Hom(\bG_{\Weil{}},\EEx)$ are continuous when $\bG$ is
viewed as the base of the espace \'etal\'e attached to $\gcs{L}$.
Since $\ell$ is invertible in $\Fq$, it follows that the image of $\Aut(\cs{L}) \to \Hom(\bG_{\Weil{}},\EEx)$ is contained in $\Hom(\pi_0(\bG_{\Weil{}}),\EEx)$. 
We also have $\pi_0(\bG_{\Weil{}})=\pi_0(\bG)_{\Weil{}}$. 
To see that $\Aut(\cs{L}) \to \Hom(\pi_0(\bG)_{\Weil{}},\EEx)$ is surjective, begin with
$\theta\in\Hom(\pi_0(\bG)_{\Weil{}},\EEx)$ and, for each $[x] \in \pi_0(\bG)_{\Weil{}}$ define
$\bar\rho^y: \gcs{L}^y \to \gcs{L}^y$ by scalar multiplication by $\theta([x])\in \EEx$ for each $y\in [x]$.
This defines an isomorphism $\bar\rho : \gcs{L}\to \gcs{L}$ of local systems on $\bG$ compatible with $\mu$ and $\phi$,
and thus an isomorphism $\rho :\cs{L}\to \cs{L}$ which maps to $\theta$ under $\Aut(\cs{L}) \to \Hom(\pi_0(\bG)_{\Weil{}},\EEx)$.
\end{proof}

\subsection{The dictionary for commutative algebraic groups over finite fields}\label{ssec:alg_groups}

Having extended the function--sheaf dictionary from connected commutative algebraic groups over $\Fq$ to smooth
commutative group schemes $G$ over $\Fq$, we look back briefly to the case when $G$ is a commutative algebraic group. 
Although Weil sheaves are not necessary in that case, the dictionary is still not perfect, generally.

\begin{corollary}\label{cor:alg_groups}
Let $G$ be a commutative algebraic group over $\Fq$.
All character sheaves on $G$ descend to $G$: $\CS(G)$ is equivalent to $\bCS(G)$ (\S \ref{ssec:bS}).
The trace of Frobenius determines a short exact sequence
\[
\begin{tikzcd}
0 \arrow{r} & \Hh^2(\pi_0(\bG),\EEx)^{\Weil{}} \arrow{r} & \bCSiso{G} \arrow{r}{\TrFrob{G}} & G(\Fq)^* \arrow{r} & 0.
\end{tikzcd}
\]
The group $\Hh^2(\pi_0(\bG),\EEx)^{\Weil{}}$ need not be trivial.
\end{corollary}

\begin{proof}
Since $G(\Fq)$ is finite, the first statement follows from Propositions~\ref{prop:BG}
 and \ref{prop:bounded}. The second statement is then a consequence of Theorem~\ref{thm:snake}. 
 The third statement is justified by Example~\ref{eg:H2}.
 \end{proof}

\begin{remark}\label{rem:Lusztig}
Suppose $G$ is a commutative algebraic group over $\Fq$ appearing in the series of papers starting with \cite{lusztig:disconnected1}.
Then $G$ is a commutative reductive algebraic group over $\Fq$ with cyclic component group scheme;
such groups are extensions of finite cyclic groups by algebraic tori.
Every Frobenius-stable character sheaf on $G$, in the sense of \cite{lusztig:disconnected1}, is a character sheaf on $G$, in our sense,
and vice versa.
Moreover, since $\Hh^2(\pi_0(\bG),\EEx)=0$ by Remark~\ref{eg:H2}, it follows that each Frobenius-stable
character sheaf on $G$ as in \cite{lusztig:disconnected1}, is determined by its trace of Frobenius, up to isomorphism. 
\end{remark}

\subsection{Base change}\label{ssec:basechange}

When using character sheaves to study characters, it is useful to understand
how character sheaves behave under change of fields.
Let $k'$ be a finite extension of $k$. Then $k \hookrightarrow k'$ induces a group homomorphism
$i_{k'/k} : G(k) \hookrightarrow G(k')$ and thus a homomorphism
\begin{align*}
i_{k'/k}^* : G(k')^* &\to G(k)^* \\
\chi &\mapsto \chi\circ i_{k'/k}.
\end{align*}
We can interpret this operation on characters in terms of character sheaves:

\begin{proposition} \label{prop:csbe}
Set $G_{k'} \ceq G\times_\Spec{k} \Spec{k'}$ and let
\[
\CS(\Res_{k'/k}(G_{k'})) \xrightarrow{\iota^*} \CS(G)
\]
be the functor obtained by pullback along the canonical closed immersion 
\[\iota : G \hookrightarrow \Res_{k'/k}(G_{k'})\] of $k$-schemes.
The following diagram commutes:
\[
\begin{tikzcd}
\CSiso{\Res_{k'/k}(G_{k'})} \arrow[two heads]{r}{\iota^*} \dar[swap]{\TrFrob{\Res_{k'/k}(G_{k'})}} & \CSiso{G} \dar{\TrFrob{G}} \\
G(k')^* \arrow[two heads]{r}{i_{k'/k}^*} & G(k)^*.
\end{tikzcd}
\]
\end{proposition}
\begin{proof}
The closed immersion $\iota : G \hookrightarrow \Res_{k'/k}(G_{k'})$ is given by \cite{bosch-lutkebohmert-reynaud:NeronModels}*{\S 7.6}.
Proposition~\ref{prop:csbe} follows immediately from Lemma~\ref{lem:pullback} together with the identifications
\[
\Res_{k'/k}(G_{k'})(k) \cong G_{k'}(k') \cong G(k')
\]
from the definitions of Weil restriction and base change.
\end{proof}

In the opposite direction, let $\Nm : G(k') \to G(k)$ be the norm map and consider the group homomorphism:
\begin{align*}
\Nm^* : G(k)^* &\to G(k')^* \\
\chi &\mapsto \chi\circ \Nm.
\end{align*}
We can also interpret this operation in terms of character sheaves.

If $\cs{L} \ceq (\gcs{L}, \mu, \phi)$ is a character sheaf on $G$, we define
$\cs{L}' \ceq (\gcs{L}, \mu, \phi_{k'})$ on the base change
$G_{k'}$ of $G$ to $k'$ by setting
\[
\phi_{k'} \ceq \phi \circ \Frob{G}^*(\phi) \circ \cdots \circ (\Frob{G}^{n-1})^*(\phi).
\]
The commutativity of the diagram (CS.3) for $\phi_{k'}$
follows from the fact that $\Frob{G_{k'}} = \Frob{G}^n$.
Note that we may also think about the construction of $\phi_{k'}$ from $\phi$
as restricting the action $\varphi$ of $\Weil{k}$ on $\gcs{L}$,
defined in Section~\ref{ssec:category}, to $\Weil{k'}$.

\begin{proposition}\label{prop:basechange}
With notation above,
the rule $\nu_{k'/k}: (\gcs{L}, \mu, \phi) \mapsto (\gcs{L}, \mu, \phi_{k'})$
 defines a monoidal functor $\CS(G) \to \CS(G_{k'})$
 such that the following diagram commutes:
\[
\begin{tikzcd}[column sep=60]
\CSiso{G} \rar{\nu_{k'/k}} \dar{\TrFrob{G}} & \CSiso{G_{k'}} \dar{\TrFrob{G_{k'}}} \\
G(k)^*  \rar{\Nm^*} & G(k')^*.
\end{tikzcd}
\]
\end{proposition}

\begin{proof}
Let $\cs{L} \ceq (\gcs{L}, \mu, \phi) \in \CS(G)$ and write $F$ for $\Frob{G}$.  For any $x \in G(k')$,
we may compute the value of $\TrFrob{G_{k'}}(\nu_{k'/k}\cs{L})(x)= t_{\nu_{k'/k}\cs{L}}(x)$ as the trace of $\phi_{k'}$ on $\gcs{L}_x$,
and the value of $\Nm^*(\TrFrob{G}(\cs{L}))(x)$ as the trace of $\phi$ on $\gcs{L}_{\Nm(x)}$.
Applying \ref{CS.3} to the stalk of $\gcs{L}^{\boxtimes n}$ at the point $(x, \Frob{}(x), \ldots, \Frob{}^{n-1}(x))$ yields a diagram
\[
\begin{tikzcd}
\gcs{L}_{\Nm(x)} \rar \dar{\phi_{\Nm(x)}} & \gcs{L}_{F(x)} \otimes \gcs{L}_{F^2(x)} \otimes \cdots \otimes \gcs{L}_x
\dar{\phi_x \otimes (F^*\phi)_x \otimes \cdots \otimes ((F^{n-1})^*\phi)_x} \\
\gcs{L}_{\Nm(x)} \rar & \gcs{L}_x \otimes \gcs{L}_{F(x)} \otimes \cdots \otimes \gcs{L}_{F^{n-1}(x)}.
\end{tikzcd}
\]
Choose a basis vector $v$ for $\gcs{L}_{\Nm(x)}$ and write $v_0 \otimes v_1 \otimes \cdots \otimes v_{n-1}$ for the image of $v$ under the
bottom map,
for $v_i \in \gcs{L}_{\Frob{}^i(x)}$.  By \ref{CS.2}, $v$ maps to
$v_1 \otimes v_2 \otimes \cdots \otimes v_0$ along the top of the diagram.
Let $\alpha_i \in \EEx$ represent $((F^i)^*\phi)_x$ with respect to these bases and let $\alpha$ be
the trace of $\phi_{\Nm(x)}$.  We may now equate the trace $\alpha$ of $\phi$ on $\gcs{L}_{\Nm(x)}$
with the product $\alpha_0 \cdots \alpha_{n-1}$, which is the trace of $\phi_{k'}$ on $\gcs{L}_x$.
\end{proof}

Finally, let $G'$ be a smooth commutative group scheme over $k'$.
We explain how to geometrize the canonical isomorphism between characters of $G'(k')$ and of $(\Res_{k'/k}G')(k)$.
We may decompose the base change $(\Res_{k'/k}G')_{k'}$ of $\Res_{k'/k}G'$ to $k'$
into a product of copies of $G'$, indexed by elements of $\Gal(k'/k)$:
\[
(\Res_{k'/k}G')_{k'} \cong \prod_{\Gal(k'/k)} G'.
\]
Since products and coproducts agree for group schemes we have a natural inclusion of $k'$-schemes
\[
G' \hookrightarrow (\Res_{k'/k}G')_{k'},
\]
mapping $G'$ into the summand corresponding to $1 \in \Gal(k'/k)$.  Composing $\nu_{k'/k}$
from Proposition~\ref{prop:basechange} with pullback along this map yields a functor
\[
\rho : \CS(\Res_{k'/k}G') \to \CS(G').
\]

\begin{proposition}
Let $k'/k$ be a finite extension and let $G'$ be a smooth commutative group scheme over $k'$.
Then the functor 
\[
\rho : \CS(\Res_{k'/k}G') \to \CS(G'),
\]
defined above, induces
\[
\begin{tikzcd}
\CSiso{\Res_{k'/k} G'} \dar{\TrFrob{\Res_{k'/k} G'}} \rar{\rho} & \CSiso{G'} \dar{\TrFrob{G'}}\\
G'(k')^* \rar & G'(k')^*,
\end{tikzcd}
\]
where the bottom map is the identity.
\end{proposition}
\begin{proof}
By Lemma~\ref{lem:pullback} the pullback part of the definition of $\rho$ corresponds to the map
\[
(\Res_{k'/k}G')(k')^* \to G'(k')^*
\]
induced by $g \mapsto (g, 1, \ldots, 1)$.  Since the action of $\Gal(k'/k)$ on
\[
(\Res_{k'/k}G')_{k'} \cong \prod_{\Gal(k'/k)} G'
\]
is given by permuting coordinates, composition with the norm map yields the identity on $G'(k')$.
\end{proof}

\section{Quasicharacter sheaves for \texorpdfstring{$p$}{p}-adic tori}\label{sec:applications}

Let $K$ be a local field with ring of integers $R$ and finite residue field $\Fq$; in this section we denote the group $\Weil{}$ by $\Weil{\Fq}$.
We continue to assume that $\ell$ is invertible in $\Fq$.

\subsection{N\'eron models}\label{ssec:neron}
We will consider connected commutative algebraic groups over $K$ that admit a N\'eron model,
by which we mean a locally finite type N\'eron model.
By \cite{bosch-lutkebohmert-reynaud:NeronModels}*{\S 10.2, Thm 2}, these are precisely the connected
commutative algebraic groups over $K$ that contain no subgroup isomorphic to $\mathbb{G}_\text{a}$.
Write $\mathcal{N}_K$ for the full subcategory of the category of algebraic groups consisting of such objects.
This category is additive, and includes all algebraic tori over $K$, abelian varieties over $K$ and unipotent $K$-wound groups.
We write $\mathcal{N}$ for the category of N\'eron models that arise in this way; in particular, $\mathcal{N}$
is a full subcategory of the category of smooth commutative group schemes over $R$.

\begin{example}
If $T_K = \Gm{K}$, then a N\'eron model can be obtained by gluing copies of $\Gm{R}$ (one for each $n \in \ZZ$) along their generic fibers,
via the gluing morphisms $\Spec{\ZZ[t, t^{-1}]} \to \Spec{\ZZ[t, t^{-1}]}$ defined by $t \mapsto \pi^n t$
\cite{bosch-lutkebohmert-reynaud:NeronModels}*{\S 10.1, Ex. 5}.

Suppose that $K'/K$ is a quadratic extension and $T_K = U_1(K'/K)$ is the unitary group.
When $K'/K$ is unramified, the N\'eron model of $T_K$ is a form of the N\'eron model for $\Gm{K}$, with the nontrivial automorphism
$\sigma \in \Gal(K'/K)$ mapping $(x, n)$ to $(\sigma(x), -n)$ for $x \in R'^\times$ and $n \in \ZZ$ specifying the copy of $\Gm{R'}$.
This example illustrates the compatibility between N\'eron models and unramified base change
\cite{bosch-lutkebohmert-reynaud:NeronModels}*{\S 10.1, Prop. 3}.

On the other hand, if $K' = K(\sqrt{\pi})$ is totally ramified over $K$ then the N\'eron model of $U_1(K'/K)$ is affine, namely
$\Spec{R[x,y] / (x^2 - \pi y^2 - 1)}$.  The special fiber is the the disjoint union of two affine lines.

Finally, we remark that if $K'/K$ is any finite extension of local fields and $X'$ is a N\'eron model for $X_{K'}$ then
$\Res_{R'/R}(X')$ is a N\'eron model for $\Res_{K'/K}(X_{K'})$ \cite{bosch-lutkebohmert-reynaud:NeronModels}*{\S 7.6, Prop. 6}.
\end{example}

\subsection{Quasicharacters}\label{ssec:quasicharacters}

Write $\m$ for the maximal ideal of $R$ and set $R_n = R/\m^{n+1}$ for every non-negative integer $n$.
Suppose $X \in \mathcal{N}$.
Note that $X(K) = X(R)$.
A \emph{quasicharacter of $X(K)$} is a group homomorphism $X(K) \to \EEx$ that factors through
$X(R) \to X(R_n)$ for some non-negative integer $n$.
We note that this definition is compatible with \cite{cassels-frohlich:AlgebraicNumberTheory}*{Ch XV, \S 2.3}.
The group of quasicharacters of $X(K)$ will be denoted by 
$\Hom_\text{}(X(K),\EEx)$
 and the subgroup of those that factor through $X(R_n)$ will be denoted by $\Hom_n(X(K),\EEx)$.
In this section we will see how to geometrize and categorify quasicharacters of $X(K)$ using character sheaves.

\subsection{Review of the Greenberg transform} \label{ssec:rev_Greenberg}

Let $K$, $R$ and $R_n$ be as above.
For each $n \in \NN$, the Greenberg functor maps schemes over $R_n$ to schemes over $\Fq$.
See \cite{bertapelle-gonzales:Greenberg} for the definition and fundamental properties of the
Greenberg functor as we use it; other useful references include
\cite{greenberg:61}, \cite{greenberg:63a},
\cite{demazure-gabriel:GroupesAlgebriques}*{V, \S 4, no. 1},
\cite{bosch-lutkebohmert-reynaud:NeronModels}*{Ch. 9, \S 6} and
\cite{nicaise-sebag:motivicSerre}*{\S 2.2}.
For any non-negative integer $n$ we will write
\[
\Gr^R_n : \Sch{R} \to \Sch{\Fq}
\]
for the functor produced by composing pullback along $\Spec{R_n} \to \Spec{R}$ with the Greenberg functor. 
This functor respects open immersions, closed immersions, \'etale morphisms, smooth morphisms and
geometric components.  Moreover, there is a canonical isomorphism
\[
\Gr^R_n(X)(\Fq) \iso X(R_n)
\]
for any scheme $X$ over $R$.

For any $n\leq m$,  the surjective ring homomorphism $R_{m} \to R_n$ determines a
natural transformation 
\[
\varrho^R_{n\leq m} : \Gr^R_{m} \to \Gr^R_n
\]
between additive functors.
Crucially, $\varrho^R_{n\leq m}(X): \Gr^R_{m}(X)\to \Gr^R_n(X)$ is an affine morphism of $\Fq$-schemes,
for every $R$-scheme $X$ and every $n\leq m$ \cite{bertapelle-gonzales:Greenberg}*{Prop 4.3}.
This observation is key to the proof that, for any scheme $X$ over $R$, the projective limit 
\[
\Gr_R(X) \ceq \varprojlim_{n\in \NN} \Gr^R_n(X),
\]
taken with respect to the surjective morphisms $\varrho^R_{n\leq m}(X) : \Gr^R_{m}(X) \to \Gr^R_n(X)$,
exists in the category of group schemes over $\Fq$;
see \cite{EGAIV3}*{\S 8.2}.
This leads to the definition of the {\it Greenberg transform}:
\[
\Gr_R : \Sch{R} \to \Sch{\Fq}.
\]
By construction, the $\Fq$-scheme $\Gr_R(X)$ comes equipped with morphisms 
\[
\varrho^R_n(X) : \Gr^R(X) \to \Gr^R_n(X),\qquad \forall n\in \NN.
\]

\subsection{Quasicharacter sheaves}
\label{ssec:CS_on_GN}
 
Set $S = \Spec{R}$ and $S_n = \Spec{R_n}$;
note that $S_0 = \Spec{\Fq}$ is the special fibre of $S$.
Let $X$ be a smooth commutative group scheme over $S$.
For every non-negative integer $n$, the Greenberg transform $\Gr^R_n(X)$ is a {\it smooth} commutative group scheme over $S_0$.
The Greenberg transform $\Gr_R(X)$ of $X$ is a commutative group scheme over $\Fq$
with $\Fq$-rational points $X(R)$.
The morphism of $\Fq$-schemes $\varrho^R_n(X) : \Gr_R(X) \to \Gr^R_n(X)$ induces a functor
\[
\varrho^R_n(X)^* : \CS(\Gr^R_n(X)) \to \CS(\Gr_R(X)),
\]
as in Lemma~\ref{lem:pullback}.

\begin{definition}\label{def:QCS}
Let $X$ be a smooth group scheme over $R$.
A {\it quasicharacter sheaf} of $X$ is a triple 
\[
\cs{F} \ceq (n, \{\cs{F}_i\}_{n\leq i}, \{\alpha_{i \le j}\}_{n\le i \le j}),
\] 
where $n$
is a non-negative integer, each $\cs{F}_i$ is a character sheaf on $\Gr^R_i(X)$ and each 
\[
\alpha_{i \le j} : \cs{L}_j \to \varrho^R_{i \le j}(X)^* \cs{L}_i
\]
 is an isomorphism; here $\alpha_{i \le i}$ is the identity and the $\alpha_{i \le j}$ are compatible with each other.  
If $\cs{F} \ceq (n, \{\cs{F}_i\}, \{\alpha_{i \le j}\})$
and $\cs{F}' \ceq (m, \{\cs{F}'_i\}, \{\alpha'_{i \le j}\})$ are objects then $\Hom(\cs{F}, \cs{F}')$ is the set
of equivalence classes of pairs $(k, \{\beta_i\}_{k \le i})$, where $n,m \le k$ and the $\beta_i : \cs{F}_i \to \cs{F}'_i$ are
morphisms of character sheaves so that
\[
\begin{tikzcd}
\cs{F}_j \rar{\alpha_j} \dar{\beta_i} & \varrho^R_{i \le j}(X)^* \cs{F}_i \dar{f_{i \le j}^*\beta_j} \\
\cs{F}'_j \rar{\alpha_j} & \varrho^R_{i \le j}(X)^* \cs{L}'_i
\end{tikzcd}
\]
commutes for all $k\le i\le j$; we identify two such pairs $(k, \{\beta_i\})$ and $(l, \{\gamma_i\})$ if $\beta_i = \gamma_i$
for sufficiently large $i$.  Identities and composites are defined in the natural way.
Let $\QCS(X)$ denote the category of quasicharacter sheaves for $X$.
\end{definition}

\begin{remark}
If $\varrho^R_n(X)^* : \CS(\Gr^R_n(X)) \to \CS(\Gr_R(X))$ is full then the construction above can be improved
by forming $\QCS(X)$ from the essential images of the functors $\varrho^R_n(X)^*$; however, we do not know if $\varrho^R_n(X)^*$ is full.
\end{remark}

\begin{remark}
We offer the following alternate construction of $\QCS(X)$.
As above, let $X$ be a smooth group scheme over $R$.
Although $\Gr_R(X)$ is not locally of finite type and therefore not smooth,
let us consider the rigid monoidal category $\CS(\Gr_R(X))$ as defined in Section~\ref{ssec:category},
though without insisting that the commutative group $\Fq$-scheme $G$ is smooth.
A quasicharacter sheaf for $X$ is an object of
the following rigid monoidal subcategory of $\CS(\Gr_R(X))$, denoted by $\QCS(X)$:
\begin{enumerate}
\item
objects in $\QCS(X)$ are the $\ell$-adic sheaves $\varrho^R_n(X)^*\cs{L}$, for $n\in \NN$ and $\cs{L} \in \CS(\Gr^R_n(X))$; 
\item
morphisms $\varrho^R_n(X)^*\cs{L} \to \varrho^R_m(X)^*\cs{L}'$ in $\QCS(X)$ are those morphisms
in $\CS(\Gr_R(X))$ which take the form $\varrho^R_m(X)^*\rho$ for $\rho \in \Hom(\varrho^R_{n\leq m}(X)^*\cs{L},\cs{L}')$
when $n\leq m$, and $\varrho^R_n(X)^*\rho$ for $\rho \in \Hom(\cs{L},\varrho^R_{m\leq n}(X)^*\cs{L}')$ when $m\leq n$.
\end{enumerate}
\end{remark}

\begin{theorem}\label{thm:CSXK}
Let $K$ be a local field with residue field $\Fq$, in which $\ell$ is invertible; 
let $R$ be the ring of integers of $K$.
\begin{enumerate}
\labitem{(0)}{x0}
The trace of Frobenius provides a natural transformation between the additive functors
\[
X \mapsto \QCSiso{X}
\qquad\text{and}\qquad
X \mapsto \Hom_{\text{}}(X(K),\EEx)
\]
as functors from $\mathcal{N}$ to the category of commutative groups.
\end{enumerate} 
Regarding this natural transformation, for every $X \in \mathcal{N}$:
\begin{enumerate}[resume]
\labitem{(1)}{x1} there is a canonical short exact sequence of commutative groups 
\[
0 \to \Hh^2(\pi_0(X)_{\bFq},\EEx)^{\Weil{\Fq}} \to \QCSiso{X} \to \Hom_\text{}(X(K),\EEx) \to 0;
\] 
\labitem{(2)}{x2} for all quasicharacter sheaves $\cs{F}$, $\cs{F}'$ on $\Gr_R(X)$, and for every
$\rho \in \Hom(\cs{F},\cs{F}')$, either $\rho$ is trivial or $\rho$ is an isomorphism;
\labitem{(3)}{x3} for all quasicharacter sheaves $\cs{F}$ for $X$, there is a canonical isomorphism
\[
\Aut(\cs{F}) \iso \Hom((\pi_0(X)_{\bFq})_{\Weil{\Fq}},\EEx)
\]
\end{enumerate}
\end{theorem}

\begin{proof}
To prove \ref{x0}, use: Proposition~\ref{prop:functorialG} with $G = \Gr^R_n(X)$; the fact that
N\'eron models are unique up to isomorphism; the fact that every $\CS(\Gr^R_n(X))$ is a full subcategory
of $\QCS(X)$; and the observation that every object in $\QCS(X)$ is in the essential image of $\CS(\Gr^R_n(X))$ for some $n$.
To prove \ref{x1}, use Theorem~\ref{thm:snake} with $G = \Gr^R_n(X)$ and then argue as in part \ref{x0}.
To prove \ref{x2}, argue as in the proof of Lemma~\ref{lem:autornaught}.
To prove \ref{x3}, use: the fact that the component group of $\Gr^R_n(X)$ is independent of $n$;
Theorem~\ref{thm:autornaught} with $G = \Gr^R_n(X)$, in which case $\pi_0(G) = \pi_0(X \times_S S_0)$
and $\pi_0(\bG) = \pi_0(X)_{\bFq}$; then argue as in part \ref{x0}.
\end{proof}

 \begin{remark}
In Section~\ref{ssec:CS_tori} we see that the \'etale site on $\Gr_R(X)$ is rich enough to geometrize all
quasicharacters of $X(K)$ as $\ell$-adic local systems on $\Gr_R(X)$, where $X$ is a N\'eron model for an
algebraic torus or an abelian variety over a local field $K$. 
It is natural to ask if the \'etale site on the generic fibre $X_K$ would have sufficed.
This seems unlikely, since the geometric \'etale fundamental group of $\Gm{K}$ is ${\hat \ZZ}$;
however, limited results in this direction were established in \cite{cunningham-kamgarpour:13a} when $K = \mathbb{Q}_p$.
\end{remark}

\subsection{Quasicharacter sheaves for \texorpdfstring{$p$}{p}-adic tori} \label{ssec:CS_tori}

As we explained in the Introduction, our original motivation for this paper was to find a
geometrization of quasicharacters of $p$-adic tori. 
This is now provided by the following adaptation of Theorem~\ref{thm:CSXK} in the case when
$T\in \mathcal{N}$ is a N\'eron model for an algebraic torus over $K$.

\begin{corollary}\label{cor:CS_tori}
Let $T$ be a N\'eron model for an algebraic torus over $K$.
The following is a commutative diagram of exact sequences.
\[
  \begin{tikzcd}[row sep=20, column sep=18]
{}  & 0 \arrow{d} & 0 \arrow{d} &  & \\ 
   & \Hh^2(X_*(T)_{\mathcal{I}_K},\EEx)^{\Weil{\Fq}}  \arrow{d} & \arrow{d} \Hh^2(X_*(T)_{\mathcal{I}_K},\EEx)^{\Weil{\Fq}} & 0 \arrow{d} & \\
    0 \rar & \QCSiso{\pi_0(T)} \arrow{r}{\pi_0^*} \dar{\TrFrob{\pi_0(T)}}
    & \QCSiso{T} \arrow{r}{\iota_0^*}  \arrow{d}{\TrFrob{\Gr_R(T)}} & \QCSiso{T^0} \arrow{r} \arrow{d}{\TrFrob{\Gr_R(T)^0}} & 0\\
    0 \arrow{r} & \Hom(\pi_0(T)(\Fq), \EEx) \arrow{r}{\text{inf}} \arrow{d}{\iso}
    & \arrow{d}{\iso} \Hom_\text{}(T(K),\EEx) \arrow{r}{\text{res}} & \Hom_\text{}(T^0(R),\EEx) \arrow{r} \arrow{d}{\iso} & 0\\
 &  0  & 0 & 0 & 
  \end{tikzcd}
 \]
 \end{corollary}
\begin{proof} 
The horizontal sequence of groups coming from categories of quasicharacter sheaves is exact by
Proposition~\ref{prop:middleexact}, together with the observation that the functors $\pi_0^*$ and $\iota^*$ preserve limits.
It is elementary that the horizontal sequence of quasicharacters is exact.
Accordingly, by Theorem~\ref{thm:CSXK}, the kernel of the  is $\Hh^2(\pi_0(T)_{\bFq},\EEx)^{\Weil{}}$. 
By \cite{bitan:discriminant}*{Eq 3.1}, the special fibre of the component group scheme for $T$ is given by
\[
 \pi_0(T)_{\bFq} = X_*(T)_{\mathcal{I}_K},
\]
where $X_*(T)$ is the cocharacter lattice of $T_K$ and $\mathcal{I}_K$ is the inertia group for $K$.
Thus,
\[
\Hh^2(\pi_0(T)_{\bFq},\EEx)^{\Weil{\Fq}} = \Hh^2(X_*(T)_{\mathcal{I}_K},\EEx)^{\Weil{\Fq}}.
\] 
Thus, the middle vertical sequence is exact.
Since $T^0$ and $\pi_0(T)$ do not lie in $\mathcal{N}$, we cannot use Theorem~\ref{thm:CSXK} to
determine the image and trace of Frobenius for these schemes. 
Instead, we observe that $T^0$ and $\pi_0(T)$ are smooth commutative group schemes over $R$,
so Definition~\ref{def:QCS} gives meaning to categories $\QCS(T^0)$ and $\QCS(\pi_0(T))$, and,
moreover, that $\pi_0(\Gr^R_n(T^0))= \pi_0(T^0)_{\Fq} = 1$ and $\pi_0(\Gr^R_n(\pi_0(T)) = \pi_0(T)_{\Fq}$ are both independent of $n$.
It follows that the vertical sequences through $\QCSiso{T^0}$ and $\QCSiso{\pi_0(T)}$ are exact by
Theorem~\ref{thm:snake} and Definition~\ref{def:QCS}.
The diagram commutes by Lemma~\ref{lem:pullback}.
\end{proof}

\begin{example}
When $T_K = \Gm{K}$ or $T_K = U_1(K'/K)$, the geometric component group $\pi_0(T)_{\bar{k}}$ is cyclic,
so $\TrFrob{\Gr_R(T)}$ is an isomorphism.  Conversely, when $T_K = \Gm{K}^2$ then
\[
\Hh^0(\Weil{\Fq}, \Hh^2(\pi_0(T), \EEx)) = \EEx,
\]
and there are uncountably many invisible quasicharacter sheaves for $T$.

We may also give examples of tori whose N\'eron models have component groups appearing in Example \ref{eg:H2}.
Let $L = K'(\sqrt{\pi})$ be a quadratic ramified extension of $K'$.  When $K = K'$ and $T_K = U_1(L/K) \times U_1(L/K)$,
the component group $\pi_0(T)_{\bar{k}}$ is $\ZZ/2\ZZ \times \ZZ/2\ZZ$ with trivial Frobenius action.  When $K'/K$ is an unramified
quadratic extension and $T_K = \Res_{K'/K} U_1(L/K')$, then $\pi_0(T)_{\bar{k}}$ is $\ZZ/2\ZZ \times \ZZ/2\ZZ$ with Frobenius
exchanging the direct factors.  Finally, let $K'/K$ be a cubic unramified extension and $S_K = \Res_{K'/K} U_1(L/K')$.  If
$T_K$ is the subtorus with character lattice $X^*(S_K) / \langle (1,1,1) \rangle$, then $\pi_0(T)_{\bar{k}}$ is $\ZZ/2\ZZ \times \ZZ/2\ZZ$
with Frobenius of order $3$.  Each of these tori will have one invisible quasicharacter sheaf.
\end{example}

We may also extract information about the automorphism groups of quasicharacter sheaves from Theorem \ref{thm:CSXK}.

\begin{corollary}
Let $T$ be a N\'eron model for an algebraic torus over $K$.
For $\cs{E}\in \QCS(\pi_0(T))$, $\cs{F}\in {\QCS(T)}$ and $\cs{F}^0\in {\QCS(T^0)}$, there are canonical isomorphisms
\[
\Aut(\cs{E})\iso (\check{T}_\ell)^{\Weil{K}},
\quad
\Aut(\cs{F})\iso (\check{T}_\ell)^{\Weil{K}},
\quad
\Aut(\cs{F}^0)\iso 1,
\]
 where $\check{T}_\ell \ceq \Hom(X_*(T),\EEx)$, the $\ell$-adic dual torus to $T_K$.
\end{corollary}

\begin{proof}
We already know $\Aut(\cs{E}) =1$ from Proposition~\ref{prop:connected}, part \ref{x3}.
By Theorem~\ref{thm:CSXK}, 
\[
\Aut(\cs{F})  \iso \Hom((\pi_0(T)_{\bFq})_{\Weil{\Fq}},\EEx).
\]
By \cite{bitan:discriminant}*{Eq 3.1} again, 
\[
\Hom((\pi_0(T)_{\bFq})_{\Weil{\Fq}},\EEx) \iso
\Hom(X_*(T)_{\Weil{K}},\EEx).
\]
But $\Hom(X_*(T)_{\Weil{K}},\EEx) \iso \Hom(X_*(T),\EEx)^{\Weil{K}}$.
So, for any quasicharacter sheaf $\cs{F}$ for $T$,
\[
\Aut(\cs{F}) \iso (\check{T}_\ell)^{\Weil{K}},
\]
canonically.
The case $X= \pi_0(T)$ is handled by the same argument, replacing Theorem~\ref{thm:CSXK} with
Theorem~\ref{thm:snake} and Definition~\ref{def:QCS}, as in the proof of Corollary~\ref{cor:CS_tori},
after observing that $\pi_0(\pi_0(T)_{\bFq}) = \pi_0(T)_{\bFq}$.
\end{proof}

\begin{remark}
Since $\pi_0(T)_{\bFq} = X_*(T)_{\mathcal{I}_K}$ by \cite{bitan:discriminant}*{Eq 3.1}, we have
\[\Hom(\pi_0(T)(\Fq), \EEx)
= \Hom((X_*(T)_{\mathcal{I}_K})^{\Weil{\Fq}}, \EEx)
= \Hom(X_*(T)_{\mathcal{I}_K}, \EEx)_{\Weil{\Fq}} 
= \Hh^1(\Weil{\Fq}, \check{T}_\ell^{\mathcal{I}_K}).
\]
By the Langlands correspondence for $p$-adic tori \cite{yu:09a}, 
\[
\Hom(T(K), \EEx) 
\iso 
\Hh^1(\Weil{K}, \check{T}_\ell),
\]
where we refer to continuous cohomology, since $\Hom(T(K), \EEx)$ refers to continuous group homomorphisms $T(K)\to \EEx$.
It now follows from the inflation-restriction exact sequence that the following diagram commutes:
\[
\begin{tikzcd}
   0 \arrow{r} & \Hom(\pi_0(T)(\Fq), \EEx) \arrow{r}{\text{inf}} \arrow{d}{\iso}
    & \arrow{d}{\iso} \Hom_\text{}(T(K),\EEx) \arrow{r}{\text{res}} & \Hom_\text{}(T^0(R),\EEx) \arrow{r} \arrow{d}{\iso} & 0\\    
    0 \arrow{r}  
 & \Hh^1(\Weil{\Fq}, \check{T}_\ell^{\mathcal{I}_K}) \arrow{r}{\text{inf}}
 & \Hh^1(\Weil{K}, \check{T}_\ell) \arrow{r}{\text{res}} 
 & \Hh^1(\mathcal{I}_K, \check{T}_\ell)^{\Weil{\Fq}} \arrow{r}  
 & 0.
\end{tikzcd}
\] 
Combining this with Corollary~\ref{cor:CS_tori}  produces the following commutative diagram of exact sequences:
\begin{equation}\label{eq:GLCT}
  \begin{tikzcd}[row sep=20, column sep=15]
{}  & 0 \arrow{d} & 0 \arrow{d} &  & \\ 
   & \Hh^2(X_*(T)_{\mathcal{I}_K},\EEx)^{\Weil{\Fq}}  \arrow{d} & \arrow{d} \Hh^2(X_*(T)_{\mathcal{I}_K},\EEx)^{\Weil{\Fq}} & 0 \arrow{d} & \\
    0 \rar & \QCSiso{\pi_0(T)} \arrow{r}{\pi_0^*} \arrow{d}
    & \QCSiso{T} \arrow{r}{\iota_0^*}  \arrow{d} & \QCSiso{T^0} \arrow{r} \arrow{d} & 0\\
   0 \arrow{r}  
 & \Hh^1(\Weil{\Fq}, \check{T}_\ell^{\mathcal{I}_K}) \arrow{r}{\text{inf}} \arrow{d}
 & \Hh^1(\Weil{K}, \check{T}_\ell) \arrow{r}{\text{res}} \arrow{d}
 & \Hh^1(\mathcal{I}_K, \check{T}_\ell)^{\Weil{\Fq}} \arrow{r} \arrow{d}  
 & 0\\
 &  0  & 0 & 0 & 
  \end{tikzcd}
 \end{equation}
It is natural to ask if the vertical surjections can be defined directly, without making use of local class field theory,
for which the results of \cite{suzuki-yoshida:12a}  and \cite{suzuki:14a} may be helpful.
The case $T_K = \Gm{K}$ is already very interesting, in which case \eqref{eq:GLCT} becomes
\begin{equation}\label{eq:Gm}
\begin{tikzcd}
    0 \rar & \QCSiso{\ZZ} \arrow{r}{\pi_0^*} \arrow{d}{\iso}
    & \QCSiso{\Gm{K}^\text{Neron}} \arrow{r}{\iota_0^*}  \arrow{d}{\iso} & \QCSiso{\Gm{R}} \arrow{r} \arrow{d}{\iso} & 0\\
    0 \arrow{r}  
 & \Hom(\Weil{\Fq}, \EEx) \arrow{r}{\text{inf}}
 & \Hom(\Weil{K}, \EEx) \arrow{r}{\text{res}} 
 & \Hom(\mathcal{I}_K, \EEx) \arrow{r}  
 & 0
\end{tikzcd}
\end{equation}
We suspect that the general case of \eqref{eq:GLCT}, where $K$ is any local field and $T_K$ is any
torus over $K$, may be deduced from \eqref{eq:Gm}.
In Section~\ref{ssec:wrK} we develop a tool for further work in that direction.
\end{remark}

\subsection{Weil restriction and quasicharacter sheaves}\label{ssec:wrK}

Let $K'/K$ be a finite Galois extension of local fields and
let $k'/k$ be the corresponding extension of residue fields.
Let $R$ and $R'$ be the rings of integers of $K$ and $K'$, respectively.
Suppose $X \in \mathcal{N}$ and set $X_{K'} \ceq X_K \times_\Spec{K} \Spec{K'}$
and let $X'$ be a N\'eron model for $X_{K'}$.

\begin{proposition}\label{prop:wrK}
The canonical closed immersion 
\[
X_K \hookrightarrow \Res_{K'/K} X_{K'}
\]
of $K$-group schemes
induces a map of $\Fq$-group schemes 
\[
f : \Gr_R(X) \to \Res_{k'/k} \Gr_{R'}(X')
\] 
which, through quasicharacter sheaves, induces
\[
\begin{tikzcd}[column sep=60]
\Hom_{\text{}}(X_K(K'), \EEx) \arrow{r}{\chi \mapsto \chi\vert_{X(K)}} &\Hom_{\text{}}(X_K(K), \EEx).
\end{tikzcd}
\]
\end{proposition}

\begin{proof}
By the N\'eron mapping property, the canonical closed immersion
\[
X_K\hookrightarrow \Res_{K'/K}(X_{K'})
\]
 extends uniquely to a morphism
\begin{equation}\label{me}
X\to \Res_{R'/R}(X')
\end{equation}
 of smooth $R$-group schemes.
Applying the functor $\Gr^R_{n}$ to \eqref{me}
and using \cite{bertapelle-gonzales:Greenberg}*{Thm 1.1} defines the morphism of smooth group schemes
\begin{equation}\label{men}
f_n: \Gr_{n-1}^R(X) \to \Res_{k'/k} \Gr_{en-1}^{R'}(X'),
\end{equation}
where $e$ is the ramification index of $K'/K$.
Using Lemma~\ref{lem:pullback}, \eqref{men} induces a functor 
\begin{equation}\label{men-pulled}
f_n^* : \CS(\Res_{k'/k} \Gr_{en-1}^{R'}(X'))\to \CS(\Gr_{n-1}^{R}(X)).
\end{equation}
Since 
\[
\left(\Res_{k'/k} \Gr_{en-1}^{R'}(X') \right)(\Fq) = \left(\Gr_{en-1}^{R'}(X')\right)(k'),
\]
it follows from Lemma~\ref{lem:pullback} that the pullback functor \eqref{men-pulled} actually induces
\[ 
\Hom_{en-1}(X'(R'),\EEx) \to \Hom_{n-1}(X(R)),\EEx)
\]
Since $X'$ is a N\'eron model for $X_{K'}$ and $X$ is a N\'eron model for $X_K$,
 this may be re-written as
 \[ 
\Hom_{en-1}(X_{K'}(K'),\EEx)= \Hom_{en-1}(X_{K}(K'),\EEx) \to \Hom_{n-1}(X_K(K)),\EEx)
\]
Passing to limits now defines
\[ 
\Hom_\text{}(X_K(K'),\EEx) \to \Hom_\text{}(X_K(K),\EEx)
\]
Argue as in Proposition~\ref{prop:csbe} to see that this is indeed restriction of characters.
\end{proof}

\subsection{Transfer of quasicharacters}\label{ssec:transfer}

Let $K$ and $L$ be local fields with rings of integers $\OK$ and $\OL$, respectively. 
Pick uniformizers $\varpi_K$ and $\varpi_L$ for $\OK$ and $\OL$, respectively;
what follows will not depend on these choices.
Suppose $\ell$ is invertible in the residue fields of $K$ and $L$.

We begin with $X_K\in \mathcal{N}_K$ with N\'eron model $X$ and $Y_{L}\in \mathcal{N}_L$ with N\'eron model $Y$.
Suppose $m$ is a positive integer such that 
\[
\OK/\varpi_K^{m}\OK \iso \OL/\varpi_L^{m}\OL.
\]
Suppose also that
\begin{equation}\label{eq:schematic_transfer}
X \times_{\Spec{\OK}} \Spec{\OK/\varpi_K^{m}\OK} \iso Y \times_{\Spec{\OL}} \Spec{\OL/\varpi_L^{m}\OL}
\end{equation}
as smooth group schemes over $\OK/\varpi_K^{m}\OK$. 
Then
\[
\Gr^{\OK}_{m-1}(X) \iso \Gr^{\OL}_{m-1}(Y)
\]
as smooth group schemes over $\Fq$.  Accordingly, by Lemma~\ref{lem:pullback},
the isomorphism above determines an equivalence of categories
\begin{equation}\label{eq:categorical_transfer}
\CS(\Gr^{\OK}_{m-1}(X)) \cong \CS(\Gr^{\OL}_{m-1}(Y))
\end{equation}
which induces an isomorphism
\begin{equation}\label{eq:transfer}
\Hom_{m-1}(X(K),\EEx)  \iso  \Hom_{m-1}(Y(L),\EEx).
\end{equation}
The isomorphism \eqref{eq:transfer} is an instance of {\it transfer} of (certain) quasicharacters between $X(K)$ and $Y(L)$. 
We now recognize this transfer of quasicharacters as a consequence of the
equivalence of categories of quasicharacter sheaves \eqref{eq:categorical_transfer}.

The isomorphism \eqref{eq:schematic_transfer} can indeed exist between quasicharacters of
algebraic tori over local fields, even when the characteristics of $K$ and $L$ differ.
Suppose $T_K$ and $M_L$ are tori over local fields $K$ and $L$,
splitting over $K'$ and $L'$, respectively.
Then $T_K$ and $M_L$ are said to be \emph{$n$-congruent} \cite{chai-yu:01a}*{\S 2} if there are isomorphisms
 \begin{align*}
  \alpha : \OO{K'}/\varpi_K^n \OO{K'} &\to \OO{L'}/\varpi_{L}^n \OO{L'} \\
  \beta : \Gal(K'/K) &\to \Gal(L'/L) \\
  \phi : X^*(T_K) &\to X^*(M_L)
 \end{align*}
 satisfying the conditions
 \begin{enumerate}
  \item $\alpha$ induces an isomorphism $\OK/\varpi_K^n \OK \to \OO{L}/\varpi_{L}^n \OO{L}$,
  \item $\alpha$ is $\Gal(K'/K)$-equivariant relative to $\beta$, and
  \item $\phi$ is $\Gal(K'/K)$-equivariant relative to $\beta$.
 \end{enumerate}
If $T_K$ and $M_L$ are $n$-congruent then $\alpha$, $\beta$ and $\phi$ determine an isomorphism 
\begin{equation}\label{transfer}
  \Hom_{n-1}(T_K(K), \EEx) \iso \Hom_{n-1}(M_L(L),\EEx).
\end{equation}
Note that if $T_K$ and $M_L$ are $n$-congruent, then they are $m$-congruent for every
$m \leq n$.

Now let $T$ be a N\'eron model for $T_K$ and let $M$ be a N\'eron model for $M_L$.
One of the main results of \cite{chai-yu:01a} gives an isomorphism of group schemes 
\[
T \times_{\Spec{\OK}} \Spec{\OK/\varpi_K^m\OK} \iso M \times_{\Spec{\OL}} \Spec{\OL/\varpi_L^m\OL}
\] 
assuming that $T_K$ and $M_L$ are sufficiently congruent.
They define a quantity $h$ (the smallest integer so that $\varpi^h$ lies in the
Jacobian ideal associated to a natural embedding of $T_K$ into an induced torus \cite{chai-yu:01a}*{\S 8.1}) and show
that if $n > 3h$ and $T_K$ and $M_L$ are $n$-congruent then there is a canonical isomorphism of smooth group schemes
 $
\Gr_{n-3h-1}(T) \to \Gr_{n-3h-1}(M)
 $
 determined by $\alpha, \beta$ and $\phi$ \cite{chai-yu:01a}*{Thm. 8.5}.
Combining this with the paragraph above gives the following instance of the geometrization of the transfer of quasicharacters.

\begin{proposition}\label{prop:transfer}
 With notation above, suppose that $T_K$ and $M_L$ are $n$-congruent and $n > 3h$.  Set $m = n-3h$.
 Then there is a canonical equivalence of categories
 \[
 \CS(\Gr^{\OK}_{m-1}(T)) \iso \CS(\Gr^{\OL}_{m-1}(M))
 \]
 determined by $\alpha, \beta$ and $\phi$ inducing an isomorphism
 \[
\Hom_{m-1}(T(K), \EEx) \iso  \Hom_{m-1}(M(L), \EEx)
 \]
compatible with \eqref{transfer}.
\end{proposition}


\end{document}